\documentclass{elsarticle}

\RequirePackage{fix-cm}
\usepackage{graphicx}
\usepackage{amsmath, amsopn,amstext,amscd,amsfonts,amssymb}
\usepackage{dsfont}
\usepackage{comment}
\usepackage[active]{srcltx}
\usepackage{graphicx, epsfig, subfig}
\usepackage{infix-RPN}

\usepackage{siunitx}
\usepackage{tikz}

\usepackage{textcomp}

\usepackage{algorithm}

\makeatletter
\def\BState{\State\hskip-\ALG@thistlm}
\makeatother

\def\downbar#1{
\setbox10=\hbox{$#1$}
            \dimen10=\ht10 \advance\dimen10 by 2.5pt
            \ifdim \dimen10<15pt 
               \advance\dimen10 by -0.5pt
               \dimen11=\dimen10
               \advance\dimen10 by 2.5pt
               \lower \dimen11
            \else \lower \ht10 \fi
            \hbox {\hskip 1.5pt \vrule height \dimen10 depth \dp10}}
\def\upbar#1{
\setbox10=\hbox{$#1$}
            \dimen10=\ht10 \advance\dimen10 by \dp10 \advance\dimen10 by 2.5pt
            \ifdim \dimen10<15pt 
                \advance\dimen10 by 2pt \fi
            \raise 2.5pt \hbox {\hskip -1.5pt \vrule height \dimen10}}

\usepackage{multicol}
\usepackage{colortbl}
\usepackage{exscale}
\usepackage{epsfig}
\usepackage{pst-plot,pst-infixplot,pstricks,graphicx}
\usepackage{amssymb,latexsym,amsthm,amsfonts,color,fancyhdr,mathrsfs}
\usepackage{lscape}
\usepackage{rotating}
\usepackage{caption}

\newrgbcolor{verde}{0.0 0.5 0.0}
\newrgbcolor{morado}{.5 0 .5}
\newrgbcolor{azul}{0 0 .5}
\newrgbcolor{rojo}{0.5 0 0}
\newrgbcolor{brown}{0.55 0.14 0.14}
\newrgbcolor{orange}{1.0 0.3 0.0}
\newrgbcolor{seagreen}{0.33 1.00 0.62}
\newrgbcolor{darkgreen}{0.00 0.39 0.00}

\newtheorem{theorem}{\bf Theorem}[section]

\newtheorem{remark}{\bf Remark}[section]

\numberwithin{equation}{section}

\journal{}

\begin{document}

\begin{frontmatter}

\title{An electrostatic interpretation of the zeros of sieved ultraspherical polynomials}

\author{K. Castillo}
\address{CMUC, Department of Mathematics, University of Coimbra,  3001-501 Coimbra, Portugal}
\ead{kenier@mat.uc.pt}

\author{M. N. de Jesus}
\address{ CI$\&$DETS/IPV, Polytechnic Institute of Viseu, ESTGV, Campus
Polit\'ecnico de Repeses, 3504-510 Viseu, Portugal}
\ead{mnasce@estv.ipv.pt}

\author{J. Petronilho}
\address{CMUC, Department of Mathematics, University of Coimbra,  3001-501 Coimbra, Portugal}
\ead{josep@mat.uc.pt}


\begin{abstract}
In a companion paper [On semiclassical orthogonal polynomials via polynomial mappings, J. Math. Anal. Appl. (2017)]
we proved that the semiclassical class of orthogonal polynomials is stable under polynomial transformations.
In this work we use this fact to derive in an unified way old and new properties concerning the sieved ultraspherical polynomials of the first and second kind. In particular we derive ordinary differential equations for these polynomials. As an application, we use the differential equation for sieved ultraspherical polynomials of the first kind to deduce that the zeros of these polynomials mark the locations of a set of particles that are in electrostatic equilibrium with respect to a particular external field.
\end{abstract}


\end{frontmatter}

\section{Introduction}

This is the second of two papers intended to develop the theory of polynomial mappings in the framework of the semiclassical orthogonal polynomial sequences. Throughout this paper we will use the abbreviations OP and OPS for orthogonal polynomial(s) and orthogonal polynomial(s) sequence(s), respectively.
In our first article \cite{KMZ1} we obtained basic properties fulfilled by monic OPS $\{p_n\}_{n\geq0}$ and $\{q_n\}_{n\geq0}$ linked by a polynomial mapping, in the sense that there exist two polynomials $\pi_k$ and $\theta_m$, of (fixed) degrees $k$ and $m$, respectively, where $0\leq m\leq k-1$, such that
$$
p_{nk+m}(x)=\theta_m(x)\,q_n(\pi_k(x))\,,\quad n=0,1,2,\ldots\;,
$$
under the assumption that one of the sequences $\{p_n\}_{n\geq0}$ or $\{q_n\}_{n\geq0}$ is a semiclassical OPS.
In particular, we proved that if at least one of the sequences $\{p_n\}_{n\geq0}$ or $\{q_n\}_{n\geq0}$ is semiclassical then so is the other one, and we gave relations between their classes \cite[Theorem 3.1]{KMZ1}.

Our present goal is to apply the results stated
in \cite{KMZ1} to the {\it sieved} OPS, introduced by Al-Salam, Allaway, and Askey \cite{AlSalamAllawayAskey}, and subsequently studied by several authors
(see e.g. \cite{Ismail-sieved1, CharrisIsmail-siev2,Ismail-sieved3,JerVan,VanAsscheMagnus,BustozIsmailWimp,CharrisIsmail, CharrisIsmailMonsalve,PacoZeCubic1,MarcioPetronilhoJAT}). The connection between sieved OPS and polynomials mappings has been observed by Charris and Ismail \cite{CharrisIsmail-siev2,CharrisIsmail}, Geronimo and Van Assche \cite{JerVan}, and Charris, Ismail, and Monsalve \cite{CharrisIsmailMonsalve}. These authors shown how the results involving sieved OPS follow by taking particular polynomial transformations.
For instance, take $\pi_k$ the monic Chebyshev polynomial of the first kind of degree $k$. Then (up to normalization) taking for $q_n$ the monic ultraspherical polynomial of degree $n$ of parameter $\lambda$ and choosing $m=0$, and so $\theta_m\equiv1$, $\{p_n\}_{n\geq0}$ becomes the monic sieved ultraspherical OPS of the first kind. Similarly, taking for $q_n$ the monic ultraspherical polynomial of degree $n$ of parameter $\lambda+1$ and choosing $m=k-1$ and $\theta_m$ the monic Chebyshev polynomial of the second kind of degree $k-1$, $\{p_n\}_{n\geq0}$ becomes the monic sieved ultraspherical OPS of the second kind.

The structure of the paper is as follows.
In Section 2 we introduce some background, including some known results on OPS
and polynomial mappings, and some basic facts on semiclassical OPS.
In Section 3 we review the definitions of the sieved ultraspherical polynomials of the first and of the second kind.
In Sections 4 and 5 we analyze separately each one of these families of OP.
For instance, we prove that both families are semiclassical of class $k-1$
except for one choice of the parameter $\lambda$ (being classical in such a case).
Using this fact and the theory of semiclassical OP presented by Maroni \cite{Maroni}, we give the structure relation that such sieved OPS satisfy, and then using these relations (together with known facts of the general theory of semiclassical OPS) we derive the linear homogeneous second order ordinary differential equation (ODE) that the sieved OP fulfill.
This ODE was obtained (by a different process) for the sieved OP of the second kind by Bustoz, Ismail, and Wimp \cite{BustozIsmailWimp}. As far as we know, the ODE for the sieved OPS of the first kind did not appeared before in the literature. The interest on such ODE comes at once from the original paper by Al-Salam, Allaway, and Askey, where in a final section devoted to some open problems concerning sieved OPS they wrote: ``A potentially very important result would be the second order differential equation these polynomials satisfy.'' In Section 6 we present an electrostatic model solved by the sieved OPS of the first kind using the second order ODE fulfilled by these polynomials.


\section{Background}

For reasons of economy of exposition, we assume familiarity with most of the results and notation appearing in
Sections 2 and 3 of our previous article \cite{KMZ1}. Let $\{p_n\}_{n\geq0}$ be a monic OPS, so that, according to Favard's theorem it is characterized by a three-term recurrence such as
\begin{equation}\label{TTRRint}
p_{n+1}(x)=(x-\beta_n)p_n(x)-\gamma_n p_{n-1}(x)\;,\quad n=0,1,2,\ldots\;,
\end{equation}
with $p_{-1}(x):=0$ and $p_0(x):=1$, where $\beta_n\in\mathbb{C}$ and
$\gamma_{n+1}\in\mathbb{C}\setminus\{0\}$ for each $n\in\mathbb{N}_0$.
In the framework of polynomial mappings, it is useful to write the recurrence relation
in terms of blocks of recurrence relations as
\begin{equation}\label{pnblock1}
\begin{array}r
(x-b_n^{(j)})p_{nk+j}(x)=p_{nk+j+1}(x)+a_n^{(j)}p_{nk+j-1}(x)\;,
\qquad \qquad \qquad \\
\rule{0pt}{1.2em} j=0,1,\dots, k-1 \, ; \quad n=0,1,2,\ldots\,.
\end{array}
\end{equation}
Without loss of generality, we assume $a_0^{(0)}:=1$.
In general, the $a_n^{(j)}$'s and $b_n^{(j)}$'s are complex numbers
with $a_n^{(j)}\neq0$ for all $n$ and $j$.
With these numbers we may construct the determinants $\Delta_n(i,j;x)$
introduced by Charris, Ismail, and Monsalve \cite{CharrisIsmail,CharrisIsmailMonsalve}, so that
\begin{equation}\label{Delt0}
\Delta_n(i,j;x):=\left\{
\begin{array}{cl}
0 & \mbox{if $j<i-2$} \\
\rule{0pt}{1.2em}
1  & \mbox{if $j=i-2$} \\
\rule{0pt}{1.5em} x-b_n^{(i-1)}  & \mbox{if $j=i-1$}
\end{array}
\right.
\end{equation}
and, if $j\geq i\geq 1$,
\begin{equation}\label{Delt1}
\Delta_n(i,j;x):=\left|
\begin{array}{cccccc}
x-b_n^{(i-1)} & 1 & 0 &  \dots & 0 & 0  \\
a_n^{(i)} & x-b_n^{(i)} & 1 &  \dots & 0 & 0 \\
0 & a_n^{(i+1)} & x-b_n^{(i+1)} &   \dots & 0 & 0 \\
\vdots & \vdots & \vdots  & \ddots & \vdots & \vdots \\
0 & 0 & 0 &  \ldots & x-b_n^{(j-1)} & 1 \\
0 & 0 & 0 &  \ldots & a_n^{(j)} & x-b_n^{(j)}
\end{array}
\right| \, ,
\end{equation}
for every $n\in\mathbb{N}_0$.
Taking into account that $\Delta_n(i,j;\cdot)$ is a polynomial
whose degree may exceed $k$, and since in (\ref{pnblock1})
the $a_n^{(j)}$'s and $b_n^{(j)}$'s were
defined only for $0\leq j\leq k-1$, we adopt the convention
\begin{equation}
b_n^{(k+j)}:=b_{n+1}^{(j)}\; ,\quad a_n^{(k+j)}:=a_{n+1}^{(j)}
\quad  i,j,n\in\mathbb{N}_0\;,
\label{convention1}
\end{equation}
and so the following useful equality holds:
\begin{equation}
\Delta_n(k+i,k+j;x)=\Delta_{n+1}(i,j;x)\;.
\label{convention2}
\end{equation}

\begin{theorem}\label{teobk1p2}\cite[Theorem 2.1]{MarcioPetronilhoJAT}
Let $\{p_n\}_{n\geq0}$ be a monic OPS characterized by the general
blocks of recurrence relations $(\ref{pnblock1})$.
Fix $r_0\in\mathbb{C}$, $k\in\mathbb{N}$, and $m\in\mathbb{N}_0$, with $0\leq m\leq k-1$ and $k\geq3$.
Then, there exist polynomials $\pi_k$ and $\theta_m$ of degrees k and m (respectively)
and a monic OPS $\{q_n\}_{n\geq0}$ such that $q_1(0)=-r_0$ and
\begin{equation}
p_{kn+m}(x)=\theta_m(x)\, q_n(\pi_k(x)) \;,\quad n=0,1,2,\ldots
\label{pnblock5p2}
\end{equation}
if and only if the following four conditions hold:
\begin{enumerate}
\item[{\rm (i)}]
$b_n^{(m)}$ is independent of $n$ for $n\geq0$;
\item[{\rm (ii)}]
$\Delta_n(m+2,m+k-1;x)$ is independent of $n$ for $n\geq0$ and for every $x$;
\item[{\rm (iii)}]
$\Delta_0(m+2,m+k-1;\cdot)$ is divisible by $\theta_m$, i.e., there
exists a polynomial $\eta_{k-1-m}$ with degree $k-1-m$ such that
$$
\Delta_0(m+2,m+k-1;x)=\theta_m(x)\,\eta_{k-1-m}(x)\, ;
$$
\item[{\rm (iv)}]
$r_n(x)$ is independent of $x$ for every $n\geq1$, where
$$
\begin{array}l
r_n(x):=
a_{n}^{(m+1)}\Delta_{n}(m+3,m+k-1;x)-a_0^{(m+1)}\Delta_{0}(m+3,m+k-1;x) \\
\rule{0pt}{1.2em} \qquad\qquad\qquad
+a_{n}^{(m)}\Delta_{n-1}(m+2,m+k-2;x)-a_0^{(m)}\Delta_{0}(1,m-2;x)\,\eta_{k-1-m}(x)\;.
\end{array}
$$
\end{enumerate}
Under such conditions, the polynomials $\theta_m$ and $\pi_k$ are explicitly given by
\begin{equation}\label{Pimka}
\begin{array}{l}
\pi_k(x)=\Delta_0(1,m;x)\,\eta_{k-1-m}(x)-a_0^{(m+1)}\,\Delta_0(m+3,m+k-1;x)+r_0
\; , \\ [0.5em]
\theta_m(x):=\Delta_0(1,m-1;x)\equiv p_m(x)\;,
\end{array}
\end{equation}
and the monic OPS $\{q_n\}_{n\geq0}$ is generated by the three-recurrence relation
\begin{equation}
q_{n+1}(x)=\left(x-r_n\right)q_{n}(x)-s_n q_{n-1}(x) \, , \quad
n=0,1,2,\ldots  \label{pnblock4p2}
\end{equation}
with initial conditions $\, q_{-1}(x)=0\,$ and $\, q_0(x)=1 \,$,
where
\begin{equation}\label{rnsn1}
r_n:=r_0+r_n(0)\; ,\quad s_n:=a_n^{(m)}a_{n-1}^{(m+1)}\cdots a_{n-1}^{(m+k-1)}\; ,
\quad n=1,2,\ldots \, .
\end{equation}
Moreover, for each $j=0,1,2,\ldots,k-1$ and all $n=0,1,2,\ldots$,
\begin{equation}
\begin{array}l
\displaystyle p_{kn+m+j+1}(x)=\frac{1}{\eta_{k-1-m}(x)}\,\left\{ \rule{0pt}{1.2em}
 \Delta_n(m+2,m+j;x)\, q_{n+1}(\pi_k(x)) \right. \qquad \\
\rule{0pt}{1.5em} \qquad\hfill \left. + \left(\prod_{i=1}^{j+1} a_n^{(m+i)}\right)
 \Delta_n(m+j+3,m+k-1;x)\, q_{n}(\pi_k(x))\,\right\} \, .
\end{array}
\label{pnblockmp2}
\end{equation}
\end{theorem}

\begin{remark}
Notice that for $j=k-1$, (\ref{pnblockmp2}) reduces to (\ref{pnblock5p2}).
\end{remark}

\begin{theorem}\label{teobk1p2measure}\cite[Theorem 3.4]{MarcioPetronilhoJAT}
Under the conditions of Theorem $\ref{teobk1p2}$, choose $r_0=0$ and assume that
$\{p_n\}_{n\geq0}$ is a monic OPS in the positive-definite sense
with respect to some positive measure $\, {\rm d}\mu\,$. Then
$\{q_n\}_{n\geq0}$ is also a monic OPS in the positive-definite sense,
orthogonal with respect to a measure $\, {\rm d}\tau\,$.
Further, assume that the following conditions hold:
\begin{itemize}
\item[{\rm (i)}] $[\xi,\eta]:={\rm co}\left({\rm supp}({\rm d}\tau)\right)$ is a compact set;
\item[{\rm (ii)}] if $m\geq1$,
$$
\int_\xi^\eta\frac{{\rm d}\tau(x)}{\,|x-\pi_k(z_i)|\,}<\infty
\quad(i=1,2,\ldots,m)\, ,
$$
where $z_1<z_2<\cdots<z_{m}$ are the zeros of $\,\theta_m\,$;
\item[{\rm (iii)}] either $\pi_k(y_{2i-1})\geq \eta $ and
$\pi_k(y_{2i})\leq \xi $ (for all possible $i$) if $k$ is odd, or
$\pi_k(y_{2i-1})\leq \xi $ and $\pi_k(y_{2i})\geq \eta $ if $k$ is
even, where $y_1<\cdots<y_{k-1}$ denote the zeros of
$\pi_k^\prime\,$;
\item[{\rm (iv)}] $\theta_m\eta_{k-1-m}$ and
$\pi_k^{\prime }$ have the same
sign at each point of the set $\pi_k^{-1}(\left[ \xi ,\eta \right] )$.
\end{itemize}
Then the Stieltjes transforms $F(\cdot;{\rm d}\mu)$ and $F(\cdot;{\rm d}\tau)$ are related by
$$
\begin{array}{r}
\displaystyle F(z;{\rm d}\mu)\,=\,\frac{-v_0\,\Delta_0(2,m-1;z)+ \left( \prod_{j=1}^m
a_0^{(j)} \right)\,\eta_{k-1-m}(z)\,F(\pi_k(z);{\rm d}\tau)}{\theta_m(z)}\, ,\qquad \\
\rule{0pt}{1.2em}
z\in\mathbb{C}\setminus\left(\,\pi_k^{-1}\left([\xi,\eta]\right)\cup\{z_1,\ldots,z_m\}\,\right)\;,
\end{array}
$$
where the normalization condition
$\;
v_0:=\int_\xi^\eta{\rm d}\tau=\int_{{\rm supp}({\rm d}\sigma)}{\rm d}\mu=:u_0\;
$
is assumed. Further, up to constant factors, the measure ${\rm d}\mu$
can be obtained from ${\rm d}\tau$ by
\begin{equation}
{\rm d}\mu(x)=\sum_{i=1}^{m}M_i\, \delta(x-z_i)\, {\rm d}x +
\left|\frac{\eta_{k-1-m}(x)}{\theta_m(x)}\right|
\frac{{\rm d}\tau(\pi_k(x))}{\pi_k^\prime(x)}\;,
\label{measure1}
\end{equation}
where if $m\geq1$
\begin{equation}
M_i:=\frac{ v_0\,\Delta_0(2,m-1;z_i)/
\left( \prod_{j=1}^m a_0^{(j)} \right)-\eta_{k-1-m}(z_i)\,F(\pi_k(z_i);{\rm
d}\tau)}{ \theta_m^\prime(z_i)}\geq0
\label{MassMi}
\end{equation}
for all $i=1,\cdots,m$.
The support of ${\rm d}\mu$ is contained in the set
$$\pi_k^{-1}\left(\, [\xi ,\eta ]\, \right)\cup\{z_1,\ldots,z_m\}\, ,$$
an union of $k$ intervals and $m$ possible mass points.
\end{theorem}

\begin{remark}\label{remarkA}
In statement (i), ${\rm co}(A)$ means the convex hull of a set $A$.
Under the conditions of Theorem \ref{teobk1p2measure},
if ${\rm d}\tau$ is an absolutely continuous
measure with density $w_{\tau}$, then the
absolutely continuous part of ${\rm d}\mu$ has density
$$
w_\mu(x):=\left|\frac{\eta_{k-1-m}(x)}{\theta_m(x)}\right|\,w_{\tau}(\pi_k(x))
$$
with support contained in an union of at most $k$ closed intervals,
and it may appear mass points at the zeros of $\theta_m$.
\end{remark}

In \cite[Section 3]{KMZ1} we stated several results concerning OPS and polynomial mappings in the framework of the theory of semiclassical OPS. In particular, in the proof of part (ii) of \cite[Theorem 3.1]{KMZ1}, we implicitly proved the following
\begin{theorem}\label{teoStiltjseries}
Under the conditions of Theorem $\ref{teobk1p2}$,
let ${\bf u}$ and ${\bf v}$ be the moment regular functionals with respect to which
$\{p_n\}_{n\geq0}$ and $\{q_n\}_{n\geq0}$ are monic OPS, respectively.
Let $S_{{\bf u}}(z):=-\sum_{n\geq0}u_n/z^{n+1}$ and $S_{{\bf v}}(z):=-\sum_{n\geq0}v_n/z^{n+1}$
(where $u_n:=\langle{\bf u},x^n\rangle$ and $v_n:=\langle{\bf v},x^n\rangle$) be the
corresponding (formal) Stieltjes series, respectively. Suppose that
there exist polynomials $\widetilde{\Phi}$, $\widetilde{C}$, and $\widetilde{D}$, such that
$$
\widetilde{\Phi}(z) S_{{\bf v}}'(z)=\widetilde{C}(z)S_{{\bf v}}(z)+\widetilde{D}(z)\;.
$$
Then $S_{{\bf u}}(z)$ fulfils
$$
\Phi_1(z)S_{{\bf u}}^\prime(z)=C_1(z)S_{{\bf u}}(z)+D_1(z)\;,
$$
where $\Phi_1$, $C_1$, and $D_1$ are polynomials given explicitly by
$$
\begin{array}{l}
\Phi_1:=v_0\theta_m\eta_{k-1-m}\sigma_{\pi_k}[\widetilde{\Phi}]\; ,\\ [0.4em]
C_1:=v_0\left(\eta_{k-1-m}^\prime\theta_m-v_0\theta_m^\prime\eta_{k-1-m}
\sigma_{\pi_k}[\widetilde{\Phi}]+\eta_{k-1-m}
\theta_m\pi_k^\prime\sigma_{\pi_k}[\widetilde{C}]\right)\; ,\\ [0.5em]
D_1:=u_0v_0\left(\Delta_0(2,m-1,\cdot)\eta_{k-1-m}^\prime-\Delta_0^\prime(2,m-1,\cdot)\eta_{k-1-m}\right)
\sigma_{\pi_k}[\widetilde{\Phi}]\\ [0.25em]
\qquad\qquad +u_0\left(\kappa_m\eta_{k-1-m}\sigma_{\pi_k}[\widetilde{D}]+v_0\Delta_0(2,m-1,\cdot)
\sigma_{\pi_k}[\widetilde{C}]\right)\eta_{k-1-m}\pi_k^\prime\;,
\end{array}
$$
and $\sigma_{\pi_k}[f](z):=f\big(\pi_k(z)\big)$ for each polynomial $f$.
\end{theorem}

\begin{remark}
A detailed study of quadratic polynomial mappings has been presented in \cite{PacoZeLAA,PacoZePM}.
Thus from now on (even if not stated explicitly) we assume that $k\geq3$.
\end{remark}

Besides the basic facts concerning semiclassical OPS given in \cite[Section 2]{KMZ1}, we recall that such families are characterized by a structure relation and a linear homogeneous second order ODE. Indeed, let $\{p_n\}_{n\geq0}$ be a monic semiclassical OPS. This means that $\{p_n\}_{n\geq0}$  is an OPS with respect to a linear functional ${\bf u}:\mathcal{P}\to\mathbb{C}$ ($\mathcal{P}$ being the space of all polynomials with complex coefficients) which fulfils a distributional differential equation of Pearson type
$$
D(\Phi{\bf u})=\Psi{\bf u}\;,
$$
where $\Phi$ and $\Psi$ are nonzero polynomials (i.e., they do not vanish identically),
and $\deg\Psi\geq1$.
According to the theory presented by Maroni in \cite{Maroni}, $\{p_n\}_{n\geq0}$ fulfills the structure relation
\begin{equation}\label{RE}
\Phi(x)p_n^\prime(x)=M_n(x)p_{n+1}(x)+N_n(x)p_n(x)\;,\quad n=0,1,2,\ldots\;,
\end{equation}
where $M_n$ and $N_n$ are polynomials that may depend of $n$,
but they have degrees (uniformly) bounded by a number independent of $n$,
which can be computed successively using the relations
\begin{equation}\label{RRMN}
\begin{array}l
N_n=-C-N_{n-1}-(x-\beta_n)M_n\\
\gamma_{n+1}M_{n+1}=-\Phi+\gamma_nM_{n-1}+(x-\beta_n)(N_{n-1}-N_n)\,,
\end{array}
\end{equation}
with initial conditions $N_{-1}:=-C$, $M_{-1}:=0$, and $M_0:=u_0^{-1}D$.
Here $\beta_n$ and $\gamma_n$ are the parameters appearing in the three-term recurrence
relation (\ref{TTRRint}), $u_0:=\langle{\bf u},1\rangle$, and $C$ and $D$ are polynomials, being $C:=\Psi-\Phi'$, and
the definition of $D$ may be seen in \cite[Section 2.2]{KMZ1}.
The structure relation (\ref{RE}) is a characteristic property of semiclassical OPS.
Another characterization of semiclassical OPS is the second order ODE
\begin{equation}\label{ED2O}
J_n(x)p_n^{\prime\prime}(x)+K_n(x)p_n^\prime(x)+L_n(x)p_n(x)=0\;,
\end{equation}
where $J_n$, $K_n$, and $L_n$ are polynomials that may depend of $n$,
but their degrees are (uniformly) bounded by a number independent of $n$.
Moreover, if $\{p_n\}_{n\geq0}$ satisfies the structure relation (\ref{RE})--(\ref{RRMN})
then $J_n$, $K_n$, and $L_n$ are given by
\begin{equation}\label{PED2O}
\begin{array}l
J_n:=\Phi M_n\\
K_n:=W\left(M_n,\Phi\right)+C M_n=\Psi M_n-\Phi M_n'\\
L_n:=W(N_n,M_n)+\left(\gamma_{n+1}M_nM_{n+1}-N_n(N_n+C)\right)M_n/\Phi\,,\\
\end{array}
\end{equation}
where $W(f,g):=f g^\prime-f^\prime g$.

\section{Sieved ultraspherical polynomials}
\label{poly-semi-sieved}

Let $\{C_n^\lambda\}_{n\geq0}$ be the ultraspherical (or Gegenbauer) OPS,
defined by the recurrence relation
$$
2(n+\lambda)xC_n^\lambda(x)=(n+1)C_{n+1}^\lambda(x)
+(n+2\lambda-1)C_{n-1}^\lambda(x)\; ,\quad n\in\mathbb{N}\,,
$$
with initial conditions $C_0^\lambda(x):=1$ and $C_1^\lambda(x):=2\lambda x$, where $\lambda\neq0$.
This definition appears in \cite[Equation (4.7.17)]{Szego}, where the condition $\lambda>-1/2$ is assumed,
so that the polynomials are orthogonal in the positive-definite sense.
If $\lambda=0$ then a compatible definition is \cite[Equation (4.7.8)]{Szego}
$$
C_0^0(x):=1\, ,\quad C_n^0(x):=T_n(x)=\frac{n}{2}
\lim_{\substack{\lambda\to0 \\ \lambda\neq0}}\frac{C_n^\lambda(x)}{\lambda}
\;,\quad n\in\mathbb{N}\;.
$$
Here we allow orthogonality with respect to a quasi-definite (or regular) functional in $\mathcal{P}$,
not necessarily positive-definite.
Therefore we assume that the range of values of the parameter $\lambda$ is
\begin{equation}\label{lambda-regular}
\lambda\in\mathbb{C}\setminus\{-n/2:n\in\mathbb{N}\}\;.
\end{equation}
(This follows e.g. from \cite[Table 1]{KMZ1}, noticing that $C_n^\lambda$ is, up to normalization, a Jacobi polynomial $P_n^{(\alpha,\beta)}$ with parameters $\alpha=\beta=\lambda-1/2$.)
We recall the definition of the sieved ultraspherical polynomials, as presented in
\cite{AlSalamAllawayAskey} and \cite{Ismail-sieved1}.
Rogers \cite{Rogers1,Rogers2} studied the OPS $\{C_n(\cdot;\beta|q)\}_{n\geq0}$ defined by
$C_0(x;\beta|q):=1$, $C_1(x;\beta|q):=2x(1-\beta)/(1-q)$, and
$$
2x(1-\beta q^n)C_n(x;\beta|q)
=(1-q^{n+1})C_{n+1}(x;\beta|q)+(1-\beta^2q^{n-1})C_{n-1}(x;\beta|q)\;,
$$
where $\beta$ and $q$ are real or complex parameters, and $|q|<1$.
Nowadays these polynomials are called continuous $q-$ultraspherical polynomials,
since they generalize $\{C_n^\lambda\}_{n\geq0}$ in the following sense (see \cite{Ismail-sieved1}):
$$
\lim_{q\to1}C_n(x;q^\lambda|q)=C_n^\lambda(x)\; .
$$
Let $\{c_n(\cdot;\beta|q)\}_{n\geq0}$ be an OPS
obtained renormalizing $\{C_n(\cdot;\beta|q)\}_{n\geq0}$, so that
$$
C_n(\cdot;\beta|q)=\frac{(\beta^2;q)_n}{(q;q)_n}\,c_n(\cdot;\beta|q)\; ,
$$
where $(a;q)_0:=1$ and $(a;q)_n:=\prod_{j=1}^n(1-aq^{j-1})$ for each $n\in\mathbb{N}$.
The sieved OP defined by Al-Salam, Allaway, and Askey \cite{AlSalamAllawayAskey}
are limiting cases of the polynomials $C_n(\cdot;\beta|q)$ and $c_n(\cdot;\beta|q)$.
Indeed, fix $k\in\mathbb{N}$ and let $\omega_k$ be an $k$th root of the unity, i.e.,
$$
\omega_k:=e^{2\pi i/k}\,. 
$$
Setting $\beta=s^{\lambda k}$ and $q=s\omega_k$,
the OPS $\{c_n^\lambda(\cdot;k)\}_{n\geq0}$ defined by
$$
c_n^\lambda(x;k):=\lim_{s\to1}c_n(x;s^{\lambda k}|s\omega_k)
$$
is the sequence of the {\it sieved ultraspherical polynomials of the first kind};
and setting $\beta=s^{\lambda k+1}\omega_k$ and $q=s\omega_k$,
the OPS $\{B_n^\lambda(\cdot;k)\}_{n\geq0}$ defined by
$$
B_n^\lambda(x;k):=\lim_{s\to1}C_n(x;s^{\lambda k+1}\omega_k|s\omega_k)
$$
is the sequence of the {\it sieved ultraspherical polynomials of the second kind}.

For $\lambda>-1/2$ the sieved ultraspherical polynomials are orthogonal in the positive-definite sense.
In such a case, the orthogonality measures
were given in \cite[Theorems 1 and 2]{AlSalamAllawayAskey}.

\section{On sieved ultraspherical OP of the second kind}

\subsection{Description via a polynomial mapping}
In \cite{CharrisIsmail-siev2}, Charris and Ismail proved that
$\{B_n^\lambda(\cdot;k)\}_{n\geq0}$ satisfies
\begin{equation}\label{sievedBn1}
B_{kn+j}^{\lambda}(x;k)=U_j(x)C_n^{\lambda+1}(T_k(x))
+U_{k-j-2}(x)C_{n-1}^{\lambda+1}(T_k(x))
\end{equation}
for $j=0,1,\ldots,k-1$ and $n=1,2,\ldots$,
where $\{T_n\}_{n\geq0}$ and $\{U_n\}_{n\geq0}$ are the OPS of the
Chebychev polynomials of the first and the second kind, respectively, defined by
$$T_n(x):=\cos(n\theta)\; ,\quad
U_n(x):=\frac{\sin(n+1)\theta}{\sin\theta}\quad (x=\cos\theta\;,\;0<\theta<\pi)\;.$$
Since $U_{-1}:=0$, then for $j=k-1$ (\ref{sievedBn1}) reduces to
\begin{equation}\label{sievedBn2}
B_{kn+k-1}^{\lambda}(x;k)=U_{k-1}(x)C_n^{\lambda+1}(T_k(x))\; ,\quad n=0,1,2,\ldots\; .
\end{equation}
Relations (\ref{sievedBn1}) and (\ref{sievedBn2}) establish a connection between sieved OP
of the second kind and OPS obtained via a polynomial mapping as described in \cite[Section 2]{KMZ1}.
This connection was established in a different way by
Geronimo and Van Assche \cite{JerVan}, and also in \cite{MarcioPetronilhoJAT}
(see also \cite{CharrisIsmailMonsalve}).
Next we briefly describe such connection following
the presentation in \cite[Section 5.2]{MarcioPetronilhoJAT}.
Taking for $\{p_n\}_{n\geq0}$ the monic OPS corresponding
to $\{B_n^\lambda(\cdot;k)\big\}_{n\geq0}$, so that
\begin{equation}\label{Sieved1}
p_{kn+j}(x)=\frac{n!}{2^{kn+j}(\lambda+1)_n}B_{kn+j}^{\lambda}(x;k) 
\end{equation}
$(n=0,1,2,\ldots\,;\;j=0,1,\ldots,k-1)$, where
$(\alpha)_n$ is the shifted factorial, defined by
$(\alpha)_0:=1$ and $(\alpha)_n:=\alpha(\alpha+1)\cdots(\alpha+n-1)$ whenever $n\geq1$,
and using the three-term recurrence relation for $\{B_n^\lambda(\cdot;k)\}_{n\geq0}$
given in \cite{AlSalamAllawayAskey}, we see that the coefficients appearing in the (block)
three-term recurrence relation (\ref{pnblock1}) for $\{p_n\}_{n\geq0}$ are
$$
\begin{array}{c}
b_{n}^{(j)}:=0\quad (0\leq j\leq k-1)\; ,\quad
a_{n}^{(j)}:=\frac14\quad (1\leq j\leq k-2)\; , \\ [0.5em]
\displaystyle
a_{n+1}^{(0)}:=\frac{n+1}{4(n+1+\lambda)} \, ,\quad
a_n^{(k-1)}:=\frac{n+1+2\lambda}{4(n+1+\lambda)} \quad
\end{array}
$$
for each $n\in\mathbb{N}_0$. Hence, for every $n\in\mathbb{N}_0$ and $0\leq j\leq k-1$, we compute
$$\Delta_n(1,j-1;x)=\widehat{U}_{j}(x)\; ,\quad
\Delta_n(j+2,k-2;x)=\widehat{U}_{k-j-2}(x)\; ,$$
where $\widehat{T}_n$ and $\widehat{U}_n$
denote the monic polynomials corresponding to $T_n$ and $U_n$,
\begin{equation}\label{TUM}
\widehat{T}_n(x):=2^{1-n}\,T_n(x)\; ,\quad
\widehat{U}_n(x):=2^{-n}U_n(x)\;,\quad n\in\mathbb{N}\;,
\end{equation}
and so one readily verifies that the hypothesis of Theorem \ref{teobk1p2}
are fulfilled, with $m=k-1$ and being the polynomial mapping described by the polynomials
\begin{equation}\label{poly-map-BnL}
\pi_k(x):=\widehat{U}_{k}(x)-\mbox{$\frac14$}\,\widehat{U}_{k-2}(x)=\widehat{T}_k(x)\, , \quad
\theta_{k-1}(x):=\widehat{U}_{k-1}(x) \,,\quad \eta_0(x):=1 \, .
\end{equation}
Moreover, $\{q_n\}_{n\geq0}$ is the monic OPS characterized by
$$
r_0=r_n=0\, , \quad s_n=4^{2-k}\, a_n^{(0)}a_n^{(k-1)}
=\frac{1}{4^k}\frac{n(n+1+2\lambda)}{(n+\lambda)(n+1+\lambda)}
\quad (n\in\mathbb{N})\;,
$$
meaning that, indeed, $q_n$ is up to an affine change of variables
the ultraspherical polynomial of degree $n$ with parameter $\lambda+1$,
\begin{equation}\label{qnCnlambda}
q_n(x)=\frac{n!}{2^{kn}(\lambda+1)_n}\, C_n^{\lambda+1}\left(2^{k-1}x\right)\;.
\end{equation}
Thus (\ref{sievedBn1}) and (\ref{sievedBn2}) follow immediately from Theorem \ref{teobk1p2}.
For $\lambda>-1/2$ the orthogonality measure for
$\{B_n^\lambda(\cdot;k)\}_{n\geq0}$ given in \cite{AlSalamAllawayAskey} may
be computed easily using Theorem \ref{teobk1p2measure},
being absolutely continuous with weight function
$$
w(x):=(1-x^2)^{\lambda+\frac12}\big|U_{k-1}(x)\big|^{2\lambda}\,,\quad -1< x<1\;.
$$
Indeed, in this situation, the masses at the zeros of $\theta_m\equiv\widehat{U}_{k-1}$ given by (\ref{MassMi}) all vanish, and so the measure given by (\ref{measure1}) becomes absolutely continuous, with a density function given by Remark \ref{remarkA}. For details, see \cite[Section 5.2]{MarcioPetronilhoJAT}.

\subsection{Classification}

According to a result by Bustoz, Ismail,
and Wimp \cite{BustozIsmailWimp}, $B_n^\lambda(\cdot;k)$ is a solution of a
linear second order ODE with polynomial coefficients, being the degrees of these polynomials
(uniformly) bounded by a number independent of $n$.
Therefore, $\{B_n^\lambda(\cdot;k)\}_{n\geq0}$ is a semiclassical OPS.
In the next theorem we state the semiclassical character of $\{B_n^\lambda(\cdot;k)\}_{n\geq0}$
in an alternative way and we give its (precise) class.
It is worth mentioning that usually the ODE is not the most efficient way
to obtain the class of a semiclassical OPS.
Often, being ${\bf u}$ the regular functional for the given (semiclassical) OPS, the
differential equation fulfilled by the corresponding
(formal) Stieltjes series $S_{{\bf u}}(z):=-\sum_{n\geq0}u_n/z^{n+1}$
allow us to obtain the class in a more simpler way.
In the next theorem we determine the class of $\{B_n^\lambda(\cdot;k)\}_{n\geq0}$ using the associated Stieltjes series and the results stated in \cite[Section 3]{KMZ1}.

\begin{theorem}\label{T3}
Let $\{p_n\}_{n\geq0}$ be the monic OPS corresponding to the sieved polynomials $\{B_{n}^{\lambda}(\cdot;k)\}_{n\geq0}$ given by (\ref{Sieved1}), being $\lambda\in\mathbb{C}\setminus\{-n/2\,:\,n\in\mathbb{N}\}$ and and $k\geq3$. Let ${\bf u}$ be the regular functional with respect to which $\{p_n\}_{n\geq0}$ is an OPS. Then
\begin{equation}\label{Eq-de-u}
D(\Phi{\bf u})=\Psi{\bf u}\;,
\end{equation}
where $\Phi$ and $\Psi$ are polynomials given by
\begin{equation}\label{pol-de-u}
\Phi(x):=(1-x^2)\widehat{U}_{k-1}(x)\;,\quad
\Psi(x):=-\big(2x\widehat{U}_{k-1}(x)+k(2\lambda+1) \widehat{T}_{k}(x)\big)\,.
\end{equation}
Moreover, the corresponding formal Stieltjes series $S_{{\bf u}}(z)$ fulfils
\begin{equation}\label{Stieltjes-monic-Bnlambda}
\Phi(z)S_{{\bf u}}^\prime(z)=C(z)S_{{\bf u}}(z)+D(z)\;,
\end{equation}
where $C$ and $D$ are polynomials given by
\begin{equation}\label{PhiC}
C(z):=-\big(z\widehat{U}_{k-1}(z)+2k\lambda\widehat{T}_k(z)\big)\,,\quad
D(z):=-2u_0\big(\widehat{U}_{k-1}(z)+k\lambda\widehat{T}_{k-1}(z)\big)\,.
\end{equation}
As a consequence, if $\lambda\in\mathbb{C}\setminus\{-n/2:n\in\mathbb{N}_0\}$
then $\{B_{n}^{\lambda}(\cdot;k)\}_{n\geq0}$ is a semiclassical OPS of class $k-1$.
If $\lambda=0$ then $\{B_{n}^{0}(\cdot;k)\}_{n\geq0}$ is (up to normalization) the
Chebychev OPS of the second kind, and so a classical OPS.
\end{theorem}

\begin{proof}
Let $\textbf{v}^{\lambda+1}$ be the regular functional
associated with the ultraspherical OPS $\{C_n^{\lambda +1}\}_{n\geq0}$,
and let $\textbf{v}$ be the regular functional associated with $\{q_n\}_{n\geq0}$
defined by (\ref{qnCnlambda}).
The relation between the corresponding formal Stieltjes series
$S_{{\bf v}}(z):=\sum_{n\geq0}v_n/z^{n+1}$ and $S_{{\bf v}^{\lambda+1}}(z):=\sum_{n\geq0}v_n^{\lambda+1}/z^{n+1}$
(where $v_n:=\langle{\bf v},x^n\rangle$ and $v_n^{\lambda+1}:=\langle{\bf v}^{\lambda+1},x^n\rangle$, $n\geq0$) is
$$S_{{\bf v}}(z)=2^{k-1}S_{{\bf v}^{\lambda+1}}\left(2^{k-1} z\right)\;.$$
Therefore, using the formal ordinary differential equation fulfilled by $S_{{\bf v}^{\lambda+1}}$
(cf. e.g. \cite{Maroni}, or see \cite[Eq. (2.4) and Table 2]{KMZ1}), we easily deduce
\begin{equation}\label{FS3}
\widetilde{\Phi}(z)S_{{\bf v}}^\prime(z)
=\widetilde{C}(z)S_{{\bf v}}(z)+\widetilde{D}(z)\,,
\end{equation}
where $\widetilde{\Phi}(x):=-x^2+4^{1-k}$, $\widetilde{C}(x):=-(2\lambda+1)x$, and
$\widetilde{D}(x):=-2(\lambda+1)v_0$.
Our aim is to prove that $\textbf{u}$ is semiclassical of class $k-1$.
Indeed, by Theorem \ref{teoStiltjseries},
\begin{equation}\label{Phi1C1D1c}
\Phi_1(z)S_{{\bf u}}^\prime(z)=C_1(z)S_{{\bf u}}(z)+D_1(z)\;,
\end{equation}
with $\Phi_1$, $C_1$, and $D_1$ given by
$$\begin{array}{l}
\Phi_1(x):=v_0\theta_{k-1}(x)\widetilde{\Phi}(\pi_k(x))\; , \\
C_1(x):=-v_0\theta_{k-1}^\prime(x)\widetilde{\Phi}(\pi_k(x))
+v_0\theta_{k-1}(x)\pi_k^\prime(x)\widetilde{C}(\pi_k(x))\\
D_1(x):=-u_0v_0\Delta_0^\prime(2,k-2,x)\widetilde{\Phi}(\pi_k(x)) \\ \qquad\qquad
+u_0\pi_k^\prime(x)\left(\big(\prod_{j=1}^{k-1}a_0^{(j)}\big)\widetilde{D}(\pi_k(x))+
v_0\Delta_0(2,k-2,x)\widetilde{C}(\pi_k(x))\right)\,.
\end{array}$$
Now, by (\ref{poly-map-BnL}) and using the elementary relations
\begin{equation}\label{CH12345}
\begin{array}{l}
\widehat{T}_n^2(x)+(1-x^2)\widehat{U}_{n-1}^2(x)=4^{1-n}\,,\quad
\widehat{U}_n^2(x)-\widehat{U}_{n-1}(x)\widehat{U}_{n+1}(x)=4^{-n}\,,\\ [0.2em]
x\widehat{U}_n(x)-(1-x^2)\widehat{U}_{n}^\prime(x)=(n+1)\widehat{T}_{n+1}(x)\,, \quad
\widehat{T}_n^\prime(x)=n\widehat{U}_{n-1}(x)\;, \\ [0.2em]
\widehat{T}_n(x)+x\widehat{U}_{n-1}(x)=2\widehat{U}_{n}(x)\,,
\end{array}
\end{equation}
after straightforward computations we deduce
$$
\begin{array}{l}
\Phi_1(x)=(1-x^2)\widehat{U}_{k-1}^3(x)\,,\quad
C_1(x)=-\widehat{U}_{k-1}^2(x)\big(x\widehat{U}_{k-1}(x)+2k\lambda\widehat{T}_k(x)\big)\,,\\ [0.3em]
D_1(x)=-2u_0\widehat{U}_{k-1}^2(x)\big(\widehat{U}_{k-1}(x)+k\lambda\widehat{T}_{k-1}(x)\big)\,.\\
\end{array}
$$
Therefore, canceling the common factor $\widehat{U}_{k-2}^2(x)$, we find that $S_\textbf{u}$ satisfies
(\ref{Stieltjes-monic-Bnlambda}), where $\Phi$, $C$, and $D$ given as in (\ref{pol-de-u}) and (\ref{PhiC}).
Since $\widehat{U}_{k-1}(\pm1)=k(\pm1)^{k-1}$, $\widehat{T}_k(\pm1)=(\pm1)^k$,
and taking into account that $\widehat{U}_{k-1}$ does not share zeros
with $\widehat{T}_k$,
we see that if $\lambda\ne0$ then the polynomials $\Phi$, $C$, and $D$ are co-prime,
hence the class of $\textbf{u}$ is equal to $s=\max\{\deg C-1, \deg D\}=k-1$.
It is clear that ${\bf u}$ satisfies (\ref{Eq-de-u}), taking into account that $\Psi(x)=C(x)+\Phi'(x)$.
If $\lambda=0$, then $\widehat{U}_{k-1}(x)$ is a common factor of the polynomials
$\Phi$, $C$, and $D$ in (\ref{PhiC}), hence canceling this factor we see that
${\bf u}$ is a classical functional, and so we see that $\{p_n\}_{n\geq0}$ is (up to normalization)
the Chebychev OPS of the second kind.
\end{proof}

\begin{remark}
Some authors define semiclassical functional requiring the pair $(\Phi,\Psi)$
appearing in the corresponding Pearson's equation to be an admissible pair, meaning that,
whenever $\deg\Phi=1+\deg\Psi$ the leading coefficient of $\Psi$ cannot
be a negative integer multiple of the leading coefficient of $\Phi$.
Medem \cite{JCMedemRoesicke} gave an example of a semiclassical functional and a corresponding pair $(\Phi,\Psi)$ which is not admissible. The above Theorem \ref{T3} shows that such a situation is not an isolated phenomenon. Indeed, choose $n_0\in\mathbb{N}$ such that $n_0+2$ is different from an integer multiple of $k$, and define
$$
\lambda:=-\frac{n_0+2+k}{2k}\;.
$$
Then, the functional ${\bf u}$ fulfilling (\ref{Eq-de-u}) is semiclassical (and so ${\bf u}$ is regular), although the corresponding pair $(\Phi,\Psi)$ given by (\ref{pol-de-u}) is not admissible. We recall, however, that for a classical functional the admissibility condition holds necessarily, a fact known as early as the work of Geronimus \cite{YaLGeronimus}.
\end{remark}

\subsection{Structure relation and second order linear ODE}

In this section we will give explicitly the structure relation
and the second order linear ODE fulfilled by the monic sieved OPS of the second kind,
given by (\ref{Sieved1}), so that
$$p_n(x)=\nu_nB_{n}^{\lambda}(x;k)\;,\quad
\nu_n:=\lfloor n/k\rfloor !/\big\{2^{n}(\lambda+1)_{\lfloor n/k\rfloor}\}\,,
$$
for each $n\in\mathbb{N}_0$, recovering in an alternative way ---in the framework of the theory of semiclassical OP---
the results given in \cite{BustozIsmailWimp}.
In what follows next we determine explicitly $M_n$ and $N_n$ for the sieved OP.

\begin{theorem}\label{Bn-MnNn}
The monic sieved OPS of the second kind 
$ p_n(x)=\nu_nB_{n}^{\lambda}(x;k)$
satisfies the structure relation $(\ref{RE})$, where
\begin{equation}\label{relStruc}
\begin{array}{rcl}
\Phi(x) &=& (1-x^2)\widehat{U}_{k-1}(x)\;, \\ [0.4em]
M_{nk+j}(x) &=& -2(nk+j+1+\lambda k)\widehat{U}_{k-1}(x) \\ [0.25em]
&&\quad -\frac{\lambda k}{2}\big(\widehat{U}_{j-1}(x)\widehat{U}_{k-j-2}(x)
-\widehat{U}_{j}(x)\widehat{U}_{k-j-3}(x)\big)\;, \\ [0.4em]
N_{nk+j}(x) &=& (nk+j+2+2\lambda k)x\widehat{U}_{k-1}(x)
-\lambda k\epsilon_j\widehat{U}_{k-2}(x) \\ [0.25em]
&&\quad +\frac{\lambda k}{8} \big(\widehat{U}_{j-1}(x)\widehat{U}_{k-j-3}(x)
-\widehat{U}_{j}(x)\widehat{U}_{k-j-4}(x)\big)
\end{array}
\end{equation}
for every $n=0,1,2,\ldots$ and $0\leq j\leq k-1$,
being $\epsilon_{k-1}:=1$, $\epsilon_{k-2}:=0$, and
$\epsilon_j:=\frac{1}{2}$ for $0\leq j\leq k-3$.
\end{theorem}

\begin{proof}
Making $m=k-1$ in (\ref{pnblockmp2}) and taking into account (\ref{poly-map-BnL}), we obtain
\begin{equation}\label{TPSieved}
p_{kn+j}(x)= \widehat{U}_j(x)q_{n}(\widehat{T}_k(x))+4^{-j}a_n^{(0)} \widehat{U}_{k-j-2}(x)q_{n-1}(\widehat{T}_k(x))
\end{equation}
Taking derivatives in both sides of (\ref{TPSieved}), we obtain
\begin{equation}\label{DSieved1}
\begin{array}l
p_{kn+j}^\prime(x)= \displaystyle
\widehat{U}_j^\prime(x)q_n(\widehat{T}_k(x))+\mathcal{A}_j(x)q_n^\prime(\widehat{T}_k(x))\\
 \rule{0pt}{1.2em}\qquad\qquad\quad  +4^{-j}a_n^{(0)} \widehat{U}_{k-j-2}^\prime(x)q_{n-1}(\widehat{T}_k(x))
  +\mathcal{B}_j(x)q_{n-1}^\prime(\widehat{T}_k(x))\,,
\end{array}\end{equation}
where $\mathcal{A}_j$ and $\mathcal{B}_j$ are polynomials defined by
$$\mathcal{A}_j(x):=\widehat{U}_j(x)\widehat{T}_k^\prime(x) \; ,\quad \mathcal{B}_j(x):=4^{-j}a_n^{(0)} \widehat{U}_{k-j-2}(x)\widehat{T}_k^\prime(x)\,.$$
Multiplying both sides of (\ref{DSieved1}) by $\widehat{U}_{k-j-2}(x)$ and using (\ref{TPSieved}),
we deduce
\begin{equation}\label{Rel1}
\begin{array}{l}
\widehat{U}_{k-j-2}(x)\left(\mathcal{A}_j(x)q_n^\prime(\widehat{T}_k(x))
+\mathcal{B}_j(x)q_{n-1}^\prime(\widehat{T}_k(x))\right) \\
\rule{0pt}{1.2em}\qquad\quad =\widehat{U}_{k-j-2}(x)p_{nk+j}^\prime(x)
-\widehat{U}_{k-j-2}^\prime(x)p_{nk+j}(x)\\
\rule{0pt}{1.2em}\qquad\qquad +\left(\widehat{U}_{k-j-2}^\prime(x)U_{j}(x)
-\widehat{U}_{k-j-2}(x)U_{j}^\prime(x)\right)q_n(\widehat{T}_k(x))\;.
\end{array}
\end{equation}
Now, since $\{q_n\}_{n\geq0}$ is a classical OPS, it fulfills the structure relation (see e.g. \cite{Maroni})
\begin{equation}\label{RE1}
\tilde{\Phi}(x)q_{n}^\prime(x)=\tilde{M}_n(x)q_{n+1}(x)+\tilde{N}_n(x)q_{n}(x)\; ,
\end{equation}
being
$\tilde{\Phi}(x)=4^{1-k}-x^2$, $\tilde{N}_n(x)=(n+2\lambda+2)x$, and $\tilde{M}_n(x)=-2(\lambda+n+1)$.
Replacing $x$ by $\widehat{T}_k(x)$ in (\ref{RE1}), and then multiplying both sides of the resulting equation by
$\mathcal{A}_j(x)\widehat{U}_{k-1}(x)\widehat{U}_{k-j-2}(x)$, one obtains a first equation.
Similarly, substituting $x$ by $\widehat{T}_k(x)$ in (\ref{RE1}), and then changing $n$ into $n-1$ and multiplying
both sides of the resulting equation by $\mathcal{B}_j(x)\widehat{U}_{k-1}(x)\widehat{U}_{k-j-2}(x)$, we obtain a second equation. Adding these two equations and using (\ref{TPSieved}) and (\ref{Rel1}), we deduce
\begin{equation}\label{relacaoP1}
\mathscr{L}_1(x)p_{nk+j}^\prime(x)=\mathscr{L}_2(x)p_{nk+j}(x)+\mathscr{L}_3(x)p_{nk+k-1}(x)+\mathscr{L}_4(x)p_{(n+1)k+k-1}(x)\,,
\end{equation}
where  $\mathscr{L}_1$, $\mathscr{L}_2$, $\mathscr{L}_3$, and $\mathscr{L}_4$ are polynomials defined by
\begin{equation}\label{Li}
\hspace*{-0.1cm}\begin{array}l
\mathscr{L}_1(x):=\widehat{U}_{k-j-2}(x)\widehat{U}_{k-1}(x)\tilde{\Phi}\big(\widehat{T}_k(x)\big)\;,\\ [0.5em]
\mathscr{L}_2(x):=\widehat{U}_{k-j-2}^\prime(x)\widehat{U}_{k-1}(x)\tilde{\Phi}\big(\widehat{T}_k(x)\big)+
\widehat{U}_{k-j-2}(x)\widehat{U}_{k-1}(x)\widehat{T}_k^\prime(x)\tilde{N}_{n-1}\big(\widehat{T}_k(x)\big)\;,\\  [0.5em]
\mathscr{L}_3(x):=\left(\widehat{U}_{k-j-2}(x)\widehat{U}_{j}^\prime(x)
-\widehat{U}_{k-j-2}^\prime(x)\widehat{U}_{j}(x)\right)\tilde{\Phi}\big(\widehat{T}_k(x)\big)\\
\rule{0pt}{1.2em}\qquad\qquad\qquad + \widehat{U}_{k-j-2}(x)\left(\mathcal{A}_j(x)\tilde{N}_{n}\big(\widehat{T}_k(x)\big)+\mathcal{B}_j(x)\tilde{M}_{n-1}
\big(\widehat{T}_k(x)\big)\right) \\
\rule{0pt}{1.2em}\qquad\qquad\qquad -
\widehat{U}_{k-j-2}(x)\widehat{U}_{j}(x)\widehat{T}_k^\prime(x)\tilde{N}_{n-1}\big(\widehat{T}_k(x)\big)\,,\\  [0.5em]
\mathscr{L}_4(x):=\mathcal{A}_j(x)\widehat{U}_{k-j-2}(x)\tilde{M}_{n}\big(\widehat{T}_k(x)\big)\;.
\end{array}
\end{equation}
Taking into account the three-term recurrence relation for $\{p_n\}_{n\geq0}$, we deduce
\begin{align}\label{relacoesP}
\nonumber p_{(n+1)k+k-1}(x)=&\big(x\widehat{U}_{k-1}(x)-a_{n+1}^{(0)}\widehat{U}_{k-2}(x)\big)p_{nk+k-1}(x) \\
 &-a_n^{(k-1)}\widehat{U}_{k-1}(x)p_{nk+k-2}(x)\; ,\\
\nonumber p_{nk+k-i}(x)=&\widehat{U}_{k-j-i-1}(x)p_{nk+j+1}(x)-\frac{1}{4}\widehat{U}_{k-j-i-2}(x)p_{nk+j}(x)
\end{align}
for every $n\in\mathbb{N}_0$ and $0\leq j\leq k-i-2$, $i=1, 2$.
Substituting (\ref{relacoesP}) in (\ref{relacaoP1}), we obtain
$$
\mathscr{L}_1(x)p_{nk+j}^\prime(x)=N_{nk+j}(x)p_{nk+j}(x)+M_{nk+j}(x)p_{nk+j+1}(x)
$$
for every $n=0,1,2, \ldots$ and $0\leq j\leq k-4$, where
\begin{equation}\label{d}
\begin{array}{l}
M_{nk+j}(x):=\mathcal{H}_1(x)\widehat{U}_{k-j-2}(x)-\mathcal{H}_2(x)\widehat{U}_{k-j-3}(x)\\
N_{nk+j}(x):=\mathscr{L}_2(x)-\frac{1}{4}\mathcal{H}_1(x)\widehat{U}_{k-j-3}(x)
+\frac{1}{4}\mathcal{H}_2(x)\widehat{U}_{k-j-4}(x)\,,
\end{array}
\end{equation}
being
$$
\begin{array}{l}
\mathcal{H}_1(x):=\mathscr{L}_3(x)
+\mathscr{L}_4(x)\big(x\widehat{U}_{k-1}(x)
-a_{n+1}^{(0)}\widehat{U}_{k-2}(x)\big)\,, \\
\mathcal{H}_2(x):=\mathscr{L}_4(x)a_n^{(k-1)}\widehat{U}_{k-1}(x)\,.
\end{array}
$$
Using some basic properties of Chebyshev polynomials we may verify that,
up to the factor $\frac{1}{k}\widehat{T}_k^\prime(x)\widehat{U}_{k-1}(x)\widehat{U}_{k-j-2}(x)$,
the relations
\begin{equation}\label{LMNj}\begin{array}{rcl}
\mathscr{L}_1(x)&=&(1-x^2)\widehat{U}_{k-1}(x)\\ [0.5em]
M_{nk+j}(x)&=&-2(nk+j+1+\lambda k)\widehat{U}_{k-1}(x) \\
&&\quad-\frac{\lambda k}{2}\left(\widehat{U}_{j-1}(x)\widehat{U}_{k-j-2}(x)
-\widehat{U}_{j}(x)\widehat{U}_{k-j-3}(x)\right)\\ [0.5em]
N_{nk+j}(x)&=&(nk+j+2+2\lambda k)x\widehat{U}_{k-1}(x)-\frac{\lambda k}{2}\widehat{U}_{k-2}(x)\\
&&\quad +\frac{\lambda k}{8}\left(\widehat{U}_{j-1}(x)\widehat{U}_{k-j-3}(x)
-\widehat{U}_{j}(x)\widehat{U}_{k-j-4}(x)\right)
\end{array}
\end{equation}
hold for every $n\in\mathbb{N}_0$ and $0\leq j\leq k-4$.
Moreover, when $j=k-1$, using the relation
$p_{nk+k-1}(x)=\widehat{U}_{k-1}(x)q_{n}(\widehat{T}_k(x))$ we may write
\begin{equation}\label{Pk-1}
\begin{array}{l}
p_{nk+k-1}^\prime(x)=\widehat{U}_{k-1}^\prime(x)q_{n}(\widehat{T}_k(x))
+\widehat{U}_{k-1}(x)\widehat{T}_k^\prime(x)q_{n}^\prime(\widehat{T}_k(x))\,.
\end{array}
\end{equation}
Multiplying both sides of (\ref{RE1}) by $\widehat{U}_{k-1}^2(x)\widehat{T}_k^\prime(x)$
and taking into account (\ref{Pk-1}) and (\ref{relacoesP}), we obtain, up to the factor
$\frac{1}{k}\widehat{T}_k^\prime(x)\widehat{U}_{k-1}(x)$,
$$
\mathscr{L}_1(x)p_{kn+k-1}^\prime(x)=N_{nk+k-1}(x)p_{nk+k-1}(x)+M_{nk+k-1}(x)p_{nk+k}(x)\;,
$$
where
\begin{equation}\label{MNk-1}
\begin{array}{l}
M_{nk+k-1}(x):=-2k(\lambda+n+1)\widehat{U}_{k-1}(x)\;,\\
N_{nk+k-1}(x):=(nk+k+2\lambda k+1)x\widehat{U}_{k-1}(x)-\lambda k\widehat{U}_{k-2}(x)\,.
\end{array}
\end{equation}
Taking into account (\ref{RRMN}), (\ref{PhiC}), and (\ref{MNk-1}), and using again some basic properties of the
Chebyshev polynomials, we deduce
\begin{align}\label{Nk-2}
\nonumber N_{nk+k-2}(x) =& -N_{nk+k-1}(x)-xM_{nk+k-1}(x)-C(x) \\
=& k(n+1+2\lambda)x\widehat{U}_{k-1}(x)\,.
\end{align}
Combining relations (\ref{RRMN}) and taking into account (\ref{PhiC}), we deduce
\begin{align*}
&\left(x^2-\mbox{$\frac14$}\right)N_{nk+k-3}(x)\\
&\,=x\big(-\Phi(x)+\mbox{$\frac14$}M_{nk+k-4}(x)+xN_{nk+k-4}(x)\big)+\mbox{$\frac14$}\left(C(x)+N_{nk+k-2}(x)\right)\\
&\,=\big(x^2-\mbox{$\frac14$}\big)\Big((k-1-nk)x\widehat{U}_{k-1}(x)-\mbox{$\frac{\lambda k}{2}$}\big(2\widehat{U}_{k-2}(x)
-4x\widehat{U}_{k-1}(x)-x\widehat{U}_{k-3}(x)\big)\Big)\,,
\end{align*}
so that
\begin{equation}\label{Nk-3}
N_{nk+k-3}(x)=\big(nk+k-1+2\lambda k\big)x\widehat{U}_{k-1}(x)-\mbox{$\frac{\lambda k}{2}$}
\widehat{U}_{k-2}(x)+\mbox{$\frac{\lambda k}{8}$}\widehat{U}_{k-4}(x)\,.
\end{equation}
Finally, using (\ref{RRMN}), (\ref{PhiC}), (\ref{Nk-2}), and (\ref{Nk-3}), we obtain
$$
\begin{array}{ll}
xM_{nk+k-2}(x)&=-N_{nk+k-3}(x)-N_{nk+k-2}(x)-C(x)\\
&=-2(nk+k-1+\lambda k)x\widehat{U}_{k-1}(x)-\frac{\lambda k}{2}x\widehat{U}_{k-3}(x)\\
xM_{nk+k-3}(x)&=-N_{nk+k-4}(x)-N_{nk+k-3}(x)-C(x)\\
&=-2(nk+k-2+\lambda k)x\widehat{U}_{k-1}(x)-\frac{\lambda k}{8}x\widehat{U}_{k-5}(x)\,,
\end{array}
$$
hence
\begin{equation}\label{Mk-2}
\begin{array}l
M_{nk+k-2}(x)=-2(nk+k-1+\lambda k)\widehat{U}_{k-1}(x)-\frac{\lambda k}{2}\widehat{U}_{k-3}(x)\\
M_{nk+k-3}(x)=-2(nk+k-2+\lambda k)\widehat{U}_{k-1}(x)-\frac{\lambda k}{8}\widehat{U}_{k-5}(x)\,.
\end{array}
\end{equation}
Thus the proof is complete.
\end{proof}

\begin{remark}
We can give alternative expressions for the polynomials $M_n$ and $N_n$
appearing in (\ref{relStruc}). Indeed, since
\begin{equation}\label{UnUmx}
\widehat{U}_{n}(x)\widehat{U}_{m}(x)-\widehat{U}_{n-1}(x)\widehat{U}_{m+1}(x)=
\left\{\begin{array}{lll}
4^{-n}\widehat{U}_{m-n}(x) & \mbox{\rm if} & 0\leq n\leq m\,;\\ [0.25em]
-4^{-m-1}\widehat{U}_{n-m-2}(x) & \mbox{\rm if} & 0\leq m< n\;,\\
\end{array}\right.
\end{equation}
we may write
$$
\begin{array}l M_{nk+j}(x)=-2(nk+j+1+\lambda
k\delta_j)\widehat{U}_{k-1}(x)-\frac{\lambda k}{2} U_{k,j}(x)\;, \\ [0.5em]
N_{nk+j}(x)=(nk+j+2+2\lambda k\delta_j)x\widehat{U}_{k-1}(x)-\frac{\lambda
k}{2}\widehat{U}_{k-2}(x)+2\lambda k V_{k,j}(x)\,, \\
\end{array}
$$
where $\delta_j:=1$ if $0\leq j\leq k-2$, $\delta_{k-1}:=0$, and
$U_{k,j}$ and $V_{k,j}$ are polynomials defined by
$$
U_{k,j}(x):=\left\{\begin{array}{lll}
-4^{-j}\widehat{U}_{k-3-2j}(x) & \mbox{\rm if} & j=0,1,\ldots,\lfloor\frac{k-3}{2}\rfloor\\ [0.25em]
4^{-k+j+2}\widehat{U}_{2j-k+1}(x) & \mbox{\rm if} & j=1+\lfloor\frac{k-3}{2}\rfloor,\ldots,
k-1\;,\\\end{array}\right.
$$
$$
V_{k,j}(x):=\left\{\begin{array}{lll}
-4^{-j-2}\widehat{U}_{k-4-2j}(x) & \mbox{\rm if} & j=0,1,\ldots,\lfloor\frac{k-4}{2}\rfloor\\ [0.25em]
4^{-k+j+1}\widehat{U}_{2j-k+2}(x) & \mbox{\rm if} & j=1+\lfloor\frac{k-4}{2}\rfloor,\ldots,
k-1\,.\\\end{array}\right.
$$
\end{remark}

\begin{remark}
Theorem \ref{Bn-MnNn} allows us to recover Theorem 3.1 in \cite{BustozIsmailWimp}.
Indeed, taking into account the three-term recurrence relation for $\{p_n\}_{n\geq0}$,
as well as (\ref{TUM}) and the first identity in (\ref{CH12345}), setting $y_n(x):=B_n^\lambda(x;k)$, we obtain
$$\left(1-T_k^2(x)\right)y_n^\prime(x)=g_{n}(x)y_{n-1}(x)+h_{n}(x)y_{n}(x)\;,$$
where $g_{n}$ and $h_{n}$ are polynomials defined by
$$
\begin{array}{l}
g_{nk+j}(x):=U_{k-1}(x)\left\{(nk+j+1+\lambda k)U_{k-1}(x)
+\lambda k\mathscr{U}_{k,j}(x)\right\}\;, \\ [0.25em]
h_{nk+j}(x):=-U_{k-1}(x)\left\{(nk+j)xU_{k-1}(x)
+\lambda kU_{k-2}(x)+\lambda k \mathscr{W}_{k,j}(x)\right\}
\end{array}
$$
for every $n\in\mathbb{N}_0$ and $0\leq j\leq k-1$,
being $\mathscr{U}_{k,j}$ and $\mathscr{W}_{k,j}$ polynomials defined by
$$\mathscr{U}_{k,j}(x):=\left\{\begin{array}{lll}
-U_{k-2j-3}(x) & \mbox{\rm if} & j=0,1,\ldots,\lfloor\frac{k-3}{2}\rfloor\\ [0.25em]
U_{2j-k+1}(x) & \mbox{\rm if} & j=1+\lfloor\frac{k-3}{2}\rfloor, \ldots,
k-1\,,\\\end{array}\right.$$
$$\mathscr{W}_{k,j}(x):=\left\{\begin{array}{lll}
-U_{k-2j-2}(x) & \mbox{\rm if} & j=0,1,\ldots,\lfloor\frac{k-2}{2}\rfloor\\ [0.25em]
U_{2j-k}(x) & \mbox{\rm if} & j=1+\lfloor\frac{k-2}{2}\rfloor, \ldots,
k-1\,.\\\end{array}\right.$$
\end{remark}
The second order linear ODE fulfilled by the sieved OPS of the second kind follows now easily.

\begin{theorem}\label{Bn-JnKnLn}
The monic sieved OPS of the second kind $p_n(x)=\nu_nB_{n}^{\lambda}(x;k)$
satisfies the second order ODE $(\ref{ED2O})$, where
\begin{equation}\label{JKL}\begin{array}{rcl}
J_{nk+j}(x)&=& \Phi(x)M_{nk+j}(x)\;,\\ [0.25em]
K_{nk+j}(x)&=& \Psi(x)M_{nk+j}(x)-\Phi(x) M_{nk+j}^\prime(x)\;,\\ [0.25em]
L_{nk+j}(x)&=& N_{nk+j}(x)M_{nk+j}^\prime(x)+\big(\Omega_j(x)-N_{nk+j}^\prime(x)\big)M_{nk+j}(x)
\end{array}\end{equation}
for all $n\geq1$ and $0\leq j\leq k-1$, being $M_{nk+j}$ and $N_{nk+j}$
given by $(\ref{relStruc})$, and
$$
\begin{array}{l}
\Phi(x):=(1-x^2)\widehat{U}_{k-1}(x)\, ,\quad
\Psi(x):=-\big(2x\widehat{U}_{k-1}(x)+k(2\lambda+1) \widehat{T}_{k}(x)\big)\;, \\ [0.25em]
\Omega_j(x) = (nk+j+1)(nk+j+2+2\lambda k)\widehat{U}_{k-1}(x)
-\frac{\lambda k}{2}\widehat{U}_{j}(x)\widehat{U}_{k-j-3}(x)\,.
\end{array}
$$
\end{theorem}

\begin{proof}
The first two equalities in (\ref{JKL}) follow immediately from (\ref{PED2O}).
To prove the third equality in (\ref{JKL}), we only need to take into account
the third equality in (\ref{PED2O}) and noticing that, using basic properties of the
Chebyshev polynomials, as well as the relations
$\widehat{U}_{m}^2(x)-\widehat{U}_{m+1}(x)\widehat{U}_{m-1}(x)=4^{-m}$ ($m=0,1,2,\ldots$),
the equality
$$\frac{a_n^{(j+1)}M_{nk+j}(x)M_{nk+j+1}(x)-N_{nk+j}(x)\big(N_{nk+j}(x)+C(x)\big)}{\Phi(x)}=\Omega_j(x)$$
holds for every $n\in\mathbb{N}_0$ and $0\leq j\leq k-1$.
\end{proof}

It is worth mentioning that a misprint appeared in the ODE given in \cite[Theorem 3.2]{BustozIsmailWimp}, as Professor Bustoz kindly commented to the third author of the present work during a visited to the Arizona State University at the 1990's.

\section{On sieved ultraspherical OP of the first kind}
\label{poly-semi-sieved}

\subsection{Description via a polynomial mapping}

Taking for $\{p_n\}_{n\geq0}$ the monic OPS corresponding
to $\{c_n^\lambda(\cdot;k)\big\}_{n\geq0}$, so that
\begin{equation}\label{Sieved1c}
p_{kn+j+1}(x)=\frac{(1+2\lambda)_n}{2^{kn+j}(\lambda+1)_n}c_{kn+j+1}^{\lambda}(x;k)
\end{equation}
$(n=0,1,2,\ldots\,;\;j=0,1,\ldots,k-1)$, and
using the three-term recurrence relation for $\{c_n^\lambda(\cdot;k)\}_{n\geq0}$
given in \cite{AlSalamAllawayAskey}, we see that the coefficients appearing in the (block)
three-term recurrence relation (\ref{pnblock1}) for $\{p_n\}_{n\geq0}$ are given by
$$
\begin{array}{c}
b_{n}^{(j)}:=0\quad (0\leq j\leq k-1)\; ,\quad
a_{n}^{(j)}:=\frac14\quad (2\leq j\leq k-1)\; , \\ [0.5em]
\displaystyle
a_{n}^{(0)}:=\frac{n}{4(n+\lambda)} \, ,\quad
a_n^{(1)}:=\frac{n+2\lambda}{4(n+\lambda)}
\end{array}
$$
for each $n\in\mathbb{N}_0$. Hence, for every $n\in\mathbb{N}_0$ and $0\leq j\leq k-1$, we compute
$$\Delta_n(2,j;x)=\widehat{U}_{j}(x)\; ,\quad
\Delta_n(j+3,k-1;x)=\widehat{U}_{k-j-2}(x)\; ,$$
and so one sees that the hypothesis of Theorem \ref{teobk1p2}
are fulfilled, with $m=0$ and being the polynomial mapping described by the polynomials
\begin{equation}\label{poly-map-cnLc}
\pi_k(x):=\widehat{U}_{k}(x)-\mbox{$\frac14$}\,\widehat{U}_{k-2}(x)=\widehat{T}_k(x)\, , \quad
\eta_{k-1}(x):=\widehat{U}_{k-1}(x) \,,\quad \theta_0(x)\equiv1 \, .
\end{equation}
Moreover, $\{q_n\}_{n\geq0}$ is the monic OPS characterized by
$$
r_0=r_n=0\, , \quad s_n=4^{2-k}\, a_n^{(0)}a_{n-1}^{(1)}
=\frac{1}{4^k}\frac{n(n-1+2\lambda)}{(n+\lambda)(n-1+\lambda)}
\quad (n\in\mathbb{N})\;,
$$
meaning that $q_n$ is up to an affine change of variables
the ultraspherical polynomial of degree $n$ with parameter $\lambda$:
\begin{equation}\label{qnCnlambdac}
q_n(x)=\frac{n!}{2^{kn}(\lambda)_n}\, C_n^{\lambda}\left(
2^{k-1}x\right)\;.
\end{equation}
For $\lambda>-1/2$, the
orthogonality measure for $\{c_n^\lambda(\cdot;k)\}_{n\geq0}$---given in \cite{AlSalamAllawayAskey}---may
be computed using Theorem \ref{teobk1p2measure},
being absolutely continuous with weight function
$$
w(x):=(1-x^2)^{\lambda-\frac12}\big|U_{k-1}(x)\big|^{2\lambda}\,,\quad -1<x<1\;.
$$

\subsection{Classification}

\begin{theorem}\label{T3sieved1}
Let $\{p_n\}_{n\geq0}$ be the monic OPS corresponding to the sieved polynomials $\{c_{n}^{\lambda}(\cdot;k)\}_{n\geq0}$, given by (\ref{Sieved1c}), being $\lambda\in\mathbb{C}\setminus\{-n/2\,:\,n\in\mathbb{N}\}$ and $k\geq3$. Let ${\bf u}$ be the regular functional with respect to which $\{p_n\}_{n\geq0}$ is an OPS. Then
\begin{equation}\label{Eq-de-uc}
D(\Phi{\bf u})=\Psi{\bf u}\;,
\end{equation}
where $\Phi$ and $\Psi$ are polynomials given by
\begin{equation}\label{pol-de-uc}
\Phi(x):=(1-x^2)\widehat{U}_{k-1}(x)\;,\quad
\Psi(x):=-k(2\lambda+1) \widehat{T}_{k}(x)\,.
\end{equation}
Moreover, the corresponding formal Stieltjes series $S_{{\bf u}}(z)$ fulfils
\begin{equation}\label{Stieltjes-monic-cnlambda}
\Phi(z)S_{{\bf u}}^\prime(z)=C(z)S_{{\bf u}}(z)+D(z)\;,
\end{equation}
where $C$ and $D$ are polynomials given by
\begin{equation}\label{PhiCc}
C(z):=z\widehat{U}_{k-1}(z)-2k\lambda\widehat{T}_k(z)\,,\quad
D(z):=-2k\lambda u_0\widehat{U}_{k-1}(z)\,.
\end{equation}
As a consequence, if $\lambda\in\mathbb{C}\setminus\{-n/2:n\in\mathbb{N}_0\}$
then $\{c_{n}^{\lambda}(\cdot;k)\}_{n\geq0}$ is a semiclassical OPS of class $k-1$.
If $\lambda=0$ then $\{c_{n}^{0}(\cdot;k)\}_{n\geq0}$ is (up to normalization) the
Chebychev OPS of the first kind, hence it is a classical OPS.
\end{theorem}

\begin{proof}
The case $\lambda=0$ is trivial, so we will assume $\lambda\neq0$.
Let $\textbf{v}^\lambda$ be the regular functional
associated with the ultraspherical OPS $\{C_n^{\lambda}\}_{n\geq0}$,
and let $\textbf{v}$ be the regular functional associated with $\{q_n\}_{n\geq0}$
defined by (\ref{qnCnlambdac}).
The relation between the corresponding formal Stieltjes series is
$$S_{\textbf{v}}(z)=2^{k-1}S_{{\textbf{v}^\lambda}}\left(2^{k-1} z\right)\;.$$
Moreover,
\begin{equation}\label{FS3}
\widetilde{\Phi}(z)S_{\textbf{v}}^\prime(z)
=\widetilde{C}(z)S_{\textbf{v}}(z)+\widetilde{D}(z)\,,
\end{equation}
where $\widetilde{\Phi}(x):=-x^2+4^{1-k}$, $\widetilde{C}(x):=-(2\lambda-1)x$, and
$\widetilde{D}(x):=-2\lambda v_0$.
Hence, by Theorem \ref{teoStiltjseries},
\begin{equation}\label{Phi1C1D1c}
\Phi_1(z)S_{{\bf u}}^\prime(z)=C_1(z)S_{{\bf u}}(z)+D_1(z)\;,
\end{equation}
where $\Phi_1$, $C_1$, and $D_1$ are given by
$$\begin{array}{l}
\Phi_1(x):=v_0\eta_{k-1}(x)\widetilde{\Phi}(\pi_k(x))\; , \\
C_1(x):=v_0\eta_{k-1}^\prime(x)\widetilde{\Phi}(\pi_k(x))
+v_0\eta_{k-1}(x)\pi_k^\prime(x)\widetilde{C}(\pi_k(x))\\
D_1(x):=u_0\eta_{k-1}^2(x)\pi_k^\prime(x)\widetilde{D}(\pi_k(x))\,.
\end{array}$$
Now, taking into account (\ref{poly-map-cnLc}), and using relations (\ref{CH12345}),
after straightforward computations and
canceling a common factor $\widehat{U}_{k-1}^2(x)$, we deduce
\begin{equation}\label{Stieltjes-monic-cnlambda}
\Phi(z)S_{{\bf u}}^\prime(z)=C(z)S_{{\bf u}}(z)+D(z)\;,
\end{equation}
where $\Phi$, $C$, and $D$ are given by (\ref{pol-de-uc}) and (\ref{PhiCc}).
Since $\widehat{U}_{k-1}(\pm1)=k(\pm1)^{k-1}$, $\widehat{T}_k(\pm1)=(\pm1)^k$,
and taking into account that $\lambda\neq0$ and $\widehat{U}_{k-1}$ does not share zeros
with $\widehat{T}_k$,
we see that the polynomials $\Phi$, $C$, and $D$ are co-prime,
hence the class of $\textbf{u}$ is equal to
$s=\max\{\deg C-1, \deg D\}=k-1$.
\end{proof}

\subsection{Structure relation and second order linear ODE}

In this section we derive the structure relation
and the second order linear ODE fulfilled by the monic sieved OPS of the first kind
given by (\ref{Sieved1c}), so that
$$p_{n+1}(x)=\vartheta_nc_{n+1}^{\lambda}(x;k)\;,\quad
\vartheta_n:=(2\lambda+1)_{\lfloor n/k\rfloor}/\big\{2^{n}(\lambda+1)_{\lfloor n/k\rfloor}\}
$$
for each $n\in\mathbb{N}_0$, and $p_0(x)\equiv1$.

\begin{theorem}\label{cn-MnNn}
The monic sieved OPS of the first kind
$ p_n(x)=\vartheta_{n-1}c_{n}^{\lambda}(x;k)$
satisfies the structure relation $(\ref{RE})$, where
\begin{equation}\label{relStruc1}
\begin{array}{rcl}
\Phi(x) &=& (1-x^2)\widehat{U}_{k-1}(x)\;, \\ [0.4em]
M_{nk+j}(x) &=& -2(nk+j+\lambda k)\widehat{U}_{k-1}(x) \\ [0.25em]
&&\quad -\frac{\lambda k}{2}\big(\widehat{U}_{j-1}(x)\widehat{U}_{k-j-2}(x)
-\widehat{U}_{j-2}(x)\widehat{U}_{k-j-1}(x)\big)\;, \\ [0.4em]
N_{nk+j}(x) &=& (nk+j+2\lambda k)x\widehat{U}_{k-1}(x)
-\lambda k\epsilon_j\widehat{U}_{k-2}(x) \\ [0.25em]
&&\quad +\frac{\lambda k}{8} \big(\widehat{U}_{j-1}(x)\widehat{U}_{k-j-3}(x)
-\widehat{U}_{j-2}(x)\widehat{U}_{k-j-2}(x)\big)
\end{array}
\end{equation}
for every $n=0,1,2,\ldots$ and $0\leq j\leq k-1$,
being $\epsilon_{k-1}:=1$, $\epsilon_{0}:=0$, and
$\epsilon_j:=\frac{1}{2}$ for $1\leq j\leq k-2$.
\end{theorem}

\begin{proof}
Making $m=0$ in (\ref{pnblockmp2}) and taking into account (\ref{poly-map-cnLc}), we obtain
\begin{equation}\label{TPSieved1c}
p_{kn+j}(x)= \mathcal{A}_j(x)q_{n+1}(\widehat{T}_k(x))+4^{1-j}a_n^{(1)} \mathcal{B}_j(x)q_{n}(\widehat{T}_k(x))
\end{equation}
for $n=0,1,2,\ldots$ and $j=1,2,\ldots,k$, where
$\mathcal{A}_j(x):=\widehat{U}_{j-1}(x)/\widehat{U}_{k-1}(x)$ and
$\mathcal{B}_j(x):=\widehat{U}_{k-j-1}(x)/\widehat{U}_{k-1}(x)$.
Taking derivatives in both sides of (\ref{TPSieved1c}), we obtain
\begin{equation}\label{DSieved1c}
\begin{array}l
p_{kn+j}^\prime(x)= \displaystyle
\mathcal{A}_j^\prime(x)q_{n+1}(\widehat{T}_k(x))+\mathcal{C}_j(x)q_{n+1}^\prime(\widehat{T}_k(x))\\
 \rule{0pt}{1.2em}\qquad\qquad\quad  +4^{1-j}a_n^{(1)} \mathcal{B}_j^\prime(x)q_{n}(\widehat{T}_k(x))
  +\mathcal{D}_j(x)q_{n}^\prime(\widehat{T}_k(x))\,,
\end{array}\end{equation}
where
$\mathcal{C}_j(x):=\mathcal{A}_j(x)\widehat{T}_k^\prime(x)$ and $\mathcal{D}_j(x):=4^{1-j}a_n^{(1)} \mathcal{B}_j(x)\widehat{T}_k^\prime(x)$.
Multiplying both sides of (\ref{DSieved1c}) by $\mathcal{B}_j(x)$ and using (\ref{TPSieved1c}),
we deduce
\begin{equation}\label{Rel1c}
\begin{array}{l}
\mathcal{B}_j(x)\left(\mathcal{C}_j(x)q_{n+1}^\prime(\widehat{T}_k(x))
+\mathcal{D}_j(x)q_{n}^\prime(\widehat{T}_k(x))\right) \\
\rule{0pt}{1.2em}\qquad\quad =\mathcal{B}_j(x)p_{nk+j}^\prime(x)
-\mathcal{B}_j^\prime(x)p_{nk+j}(x)\\
\rule{0pt}{1.2em}\qquad\qquad +\left(\mathcal{A}_j(x)\mathcal{B}_j^\prime(x)
-\mathcal{A}_j^\prime(x)\mathcal{B}_j(x)\right)q_{n+1}(\widehat{T}_k(x))\;.
\end{array}
\end{equation}
Now, since $\{q_n\}_{n\geq0}$ is a classical OPS, it fulfills the structure relation (see e.g. \cite{Maroni})
\begin{equation}\label{RE1c}
\tilde{\Phi}(x)q_{n}^\prime(x)=\tilde{M}_n(x)q_{n+1}(x)+\tilde{N}_n(x)q_{n}(x)\; ,
\end{equation}
being
$\tilde{\Phi}(x)=4^{1-k}-x^2$, $\tilde{N}_n(x)=(n+2\lambda)x$, and $\tilde{M}_n(x)=-2(\lambda+n)$.
Substituting $x$ by $\widehat{T}_k(x)$ in (\ref{RE1c}), and then multiplying both sides of the resulting equation by
$\mathcal{B}_j(x)\mathcal{D}_j(x)$, one obtains a certain equation.
Similarly, substituting $x$ by $\widehat{T}_k(x)$ in (\ref{RE1c}), and then changing $n$ into $n+1$ and multiplying
both sides of the resulting equation by $\mathcal{B}_j(x)\mathcal{C}_j(x)$, we obtain a second equation. Adding these two equations and using (\ref{TPSieved1c}) and (\ref{Rel1c}), we deduce
\begin{equation}\label{relacaoP1c}
\mathscr{S}_1(x)p_{nk+j}^\prime(x)=\mathscr{S}_2(x)p_{nk+j}(x)+\mathscr{S}_3(x)p_{(n+1)k}(x)+\mathscr{S}_4(x)p_{(n+2)k}(x)\,,
\end{equation}
where  $\mathscr{S}_1$, $\mathscr{S}_2$, $\mathscr{S}_3$, and $\mathscr{S}_4$ are polynomials defined by
\begin{equation}\label{Lic}
\hspace*{-0.1cm}\begin{array}l
\mathscr{S}_1(x):=\mathcal{B}_j(x)\tilde{\Phi}\big(\widehat{T}_k(x)\big)\;,\\ [0.5em]
\mathscr{S}_2(x):=\mathcal{B}_j^\prime(x)\tilde{\Phi}\big(\widehat{T}_k(x)\big)+
\mathcal{B}_j(x)\widehat{T}_k^\prime(x)\tilde{N}_{n}\big(\widehat{T}_k(x)\big)\;,\\  [0.5em]
\mathscr{S}_3(x):=\left(\mathcal{A}_j^\prime(x)\mathcal{B}_j(x)-\mathcal{A}_j(x)\mathcal{B}_j^\prime(x)\right)\tilde{\Phi}\big(\widehat{T}_k(x)\big)\\
\rule{0pt}{1.2em}\qquad\qquad\qquad + \mathcal{B}_j(x)\left(\mathcal{C}_j(x)\tilde{N}_{n+1}\big(\widehat{T}_k(x)\big)+\mathcal{D}_j(x)\tilde{M}_{n}
\big(\widehat{T}_k(x)\big)\right) \\
\rule{0pt}{1.2em}\qquad\qquad\qquad -
\mathcal{A}_j(x)\mathcal{B}_j(x)\widehat{T}_k^\prime(x)\tilde{N}_{n}\big(\widehat{T}_k(x)\big)\,,\\  [0.5em]
\mathscr{S}_4(x):=\mathcal{B}_j(x)\mathcal{C}_j(x)\tilde{M}_{n+1}\big(\widehat{T}_k(x)\big)\;.
\end{array}
\end{equation}
Taking into account the three-term recurrence relation for $\{p_n\}_{n\geq0}$, we deduce
\begin{equation}\label{relacoesPc}
\begin{array}{rcl}
p_{(n+2)k}(x)&=&\big(x\widehat{U}_{k-1}(x)-a_{n+1}^{(1)}\widehat{U}_{k-2}(x)\big)p_{(n+1)k}(x) \\ [0.25em]
&&\quad -a_{n+1}^{(0)}\widehat{U}_{k-1}(x)p_{(n+1)k-1}(x)\; ,\\ [0.5em]
p_{nk+k-i}(x)&=&\widehat{U}_{k-j-i-1}(x)p_{nk+j+1}(x)-\frac{1}{4}\widehat{U}_{k-j-i-2}(x)p_{nk+j}(x)
\end{array}
\end{equation}
for every $n\in\mathbb{N}_0$ and $1\leq j\leq k-i-1$, $i=0, 1$.
Substituting (\ref{relacoesPc}) in (\ref{relacaoP1c}), we obtain
$$
\mathscr{S}_1(x)p_{nk+j}^\prime(x)=N_{nk+j}(x)p_{nk+j}(x)+M_{nk+j}(x)p_{nk+j+1}(x)
$$
for every $n=0,1,2, \ldots$ and $1\leq j\leq k-2$, where
\begin{equation}\label{dc}
\begin{array}{l}
M_{nk+j}(x):=\mathcal{K}_1(x)\widehat{U}_{k-j-1}(x)-\mathcal{K}_2(x)\widehat{U}_{k-j-2}(x)\\
N_{nk+j}(x):=\mathscr{S}_2(x)-\frac{1}{4}\mathcal{K}_1(x)\widehat{U}_{k-j-2}(x)
+\frac{1}{4}\mathcal{K}_2(x)\widehat{U}_{k-j-3}(x)\,,
\end{array}
\end{equation}
being
$$
\begin{array}{l}
\mathcal{K}_1(x):=\mathscr{S}_3(x)
+\mathscr{S}_4(x)\big(x\widehat{U}_{k-1}(x)
-a_{n+1}^{(1)}\widehat{U}_{k-2}(x)\big)\,, \\
\mathcal{K}_2(x):=\mathscr{S}_4(x)a_{n+1}^{(0)}\widehat{U}_{k-1}(x)\,.
\end{array}
$$
Using some basic properties of Chebyshev polynomials we may verify that,
up to the factor $\widehat{U}_{k-j-1}(x)$, the relations
\begin{equation}\label{LMNjc}\begin{array}{rcl}
\mathscr{S}_1(x)&=&(1-x^2)\widehat{U}_{k-1}(x)\\ [0.5em]
M_{nk+j}(x)&=&-2(nk+j+\lambda k)\widehat{U}_{k-1}(x) \\
&&\quad-\frac{\lambda k}{2}\left(\widehat{U}_{j-1}(x)\widehat{U}_{k-j-2}(x)
-\widehat{U}_{j-2}(x)\widehat{U}_{k-j-1}(x)\right)\\ [0.5em]
N_{nk+j}(x)&=&(nk+j+2\lambda k)x\widehat{U}_{k-1}(x)-\frac{\lambda k}{2}\widehat{U}_{k-2}(x)\\
&&\quad +\frac{\lambda k}{8}\left(\widehat{U}_{j-1}(x)\widehat{U}_{k-j-3}(x)
-\widehat{U}_{j-2}(x)\widehat{U}_{k-j-2}(x)\right)
\end{array}
\end{equation}
hold for every $n\in\mathbb{N}_0$ and $1\leq j\leq k-2$.
Moreover, taking $j=k$ in (\ref{TPSieved1c}) and then changing $n+1$ into $n$,
we obtain $p_{nk}(x)=q_{n}(\widehat{T}_k(x))$, hence
\begin{equation}\label{Pk-1c}
\begin{array}{l}
p_{nk}^\prime(x)=\widehat{T}_k^\prime(x)q_{n}^\prime(\widehat{T}_k(x))\,.
\end{array}
\end{equation}
Substituting $x$ by $\widehat{T}_k(x)$ in (\ref{RE1c}) and multiplying both sides of (\ref{RE1c}) by $\widehat{T}_k^\prime(x)$
and taking into account (\ref{Pk-1c}) and (\ref{relacoesPc}), we obtain, up to the factor
$\widehat{U}_{k-1}(x)$,
$$
\mathscr{S}_1(x)p_{kn}^\prime(x)=M_{nk}(x)p_{nk+1}(x)+N_{nk}(x)p_{nk}(x)\;,
$$
where
\begin{equation}\label{MNk-1c}
M_{nk}(x):=-2k(\lambda+n)\widehat{U}_{k-1}(x)\;,\quad
N_{nk}(x):=k(n+2\lambda)x\widehat{U}_{k-1}(x)\,.
\end{equation}
Taking into account (\ref{RRMN}), (\ref{PhiCc}), (\ref{LMNjc}), and (\ref{MNk-1c}),
and using again some basic properties of the
Chebyshev polynomials, we deduce
\begin{equation}\label{Nk-2c}
\begin{array}{rcl}
N_{nk+k-1}(x)
&=& \mbox{$\frac{1}{x}$}\big(-\Phi(x)+\frac{1}{4}M_{nk+k-2}(x)-a_{n+1}^{(0)}M_{(n+1)k}(x)\big)+N_{nk+k-2}(x) \\ [0.25em]
&=&\big(nk+k-1+2\lambda k\big)x\widehat{U}_{k-1}(x)-\lambda k \widehat{U}_{k-2}(x)\,.
\end{array}
\end{equation}
Finally, taking into account (\ref{RRMN}), (\ref{PhiCc}), (\ref{LMNjc}), and (\ref{Nk-2c}), we obtain
$$
\begin{array}{ll}
xM_{nk+k-1}(x)&=-N_{nk+k-1}(x)-N_{nk+k-2}(x)-C(x)\\
&=-2(nk+k-1+\lambda k)x\widehat{U}_{k-1}(x)+\frac{\lambda k}{2}x\widehat{U}_{k-3}(x)\\
\end{array}$$
hence
\begin{equation}\label{Mk-2c}
\begin{array}l
M_{nk+k-1}(x)=-2(nk+k-1+\lambda k)\widehat{U}_{k-1}(x)+\frac{\lambda k}{2}\widehat{U}_{k-3}(x)
\,.
\end{array}
\end{equation}
Thus the proof is complete.
\end{proof}

\begin{remark}
We can give alternative expressions for the polynomials $M_n$ and $N_n$
appearing in (\ref{relStruc1}). Indeed, taking into account (\ref{UnUmx}),
we may write
$$
\begin{array}l M_{nk+j}(x)=-2(nk+j+\lambda
k\delta_j)\widehat{U}_{k-1}(x)-\frac{\lambda k}{2} U_{k,j}(x)\;, \\ [0.5em]
N_{nk+j}(x)=(nk+j+2\lambda k)x\widehat{U}_{k-1}(x)-\frac{\lambda
k}{2}\widehat{U}_{k-2}(x)+\frac{\lambda k}{2} V_{k,j}(x)\,, \\
\end{array}
$$
where $\delta_j:=1$ if $1\leq j\leq k-1$, $\delta_{0}:=0$, and
$U_{k,j}$ and $V_{k,j}$ are polynomials defined by
$$
U_{k,j}(x):=\left\{\begin{array}{lll}
4^{1-j}\widehat{U}_{k-1-2j}(x) & \mbox{\rm if} & j=0,1,\ldots,\lfloor\frac{k-1}{2}\rfloor\\ [0.25em]
-4^{-k+j+1}\widehat{U}_{2j-k-1}(x) & \mbox{\rm if} & j=1+\lfloor\frac{k-1}{2}\rfloor,\ldots,
k-1\;,\\\end{array}\right.
$$
$$
V_{k,j}(x):=\left\{\begin{array}{lll}
4^{-j}\widehat{U}_{k-2-2j}(x) & \mbox{\rm if} & j=0,1,\ldots,\lfloor\frac{k-2}{2}\rfloor\\ [0.25em]
-4^{-k+j+1}\widehat{U}_{2j-k}(x) & \mbox{\rm if} & j=1+\lfloor\frac{k-2}{2}\rfloor,\ldots,
k-1\,.\\\end{array}\right.
$$
\end{remark}

\medskip

\begin{theorem}\label{cn-JnKnLn}
The monic sieved OPS of the first kind
$p_n(x)=\vartheta_{n-1}c_{n}^{\lambda}(x;k)$
satisfies the second order ODE $(\ref{ED2O})$, where
\begin{equation}\label{JKL1}\begin{array}{rcl}
J_{nk+j}(x)&=& \Phi(x)M_{nk+j}(x)\;,\\ [0.25em]
K_{nk+j}(x)&=& \Psi(x)M_{nk+j}(x)-\Phi(x) M_{nk+j}^\prime(x)\;,\\ [0.25em]
L_{nk+j}(x)&=& N_{nk+j}(x)M_{nk+j}^\prime(x)+\big(\Omega_j(x)-N_{nk+j}^\prime(x)\big)M_{nk+j}(x)
\end{array}\end{equation}
for all $n\geq1$ and $0\leq j\leq k-1$, being $M_{nk+j}$ and $N_{nk+j}$ given by $(\ref{relStruc})$, and
\begin{equation}\label{PhiPsi-cn}
\begin{array}{l}
\Phi(x):=(1-x^2)\widehat{U}_{k-1}(x)\, ,\quad
\Psi(x):=-k(2\lambda+1) \widehat{T}_{k}(x)\;, \\ [0.25em]
\Omega_j(x) = (nk+j+1)(nk+j+2\lambda k)\widehat{U}_{k-1}(x)
+\frac{\lambda k}{2}\widehat{U}_{j-1}(x)\widehat{U}_{k-j-2}(x)\,.
\end{array}
\end{equation}
\end{theorem}

\begin{proof}
The first two equalities in (\ref{JKL1}) follow immediately from (\ref{PED2O}).
To prove the third equality in (\ref{JKL1}), we only need to take into account
the third equality in (\ref{PED2O}) and noticing that, using basic properties of the
Chebyshev polynomials, as well as the relations
$\widehat{U}_{m}^2(x)-\widehat{U}_{m+1}(x)\widehat{U}_{m-1}(x)=4^{-m}$ ($m=0,1,2,\ldots$),
the equality
$$\frac{a_n^{(j+1)}M_{nk+j}(x)M_{nk+j+1}(x)-N_{nk+j}(x)\big(N_{nk+j}(x)+C(x)\big)}{\Phi(x)}=\Omega_j(x)$$
holds for every $n\in\mathbb{N}_0$ and $0\leq j\leq k-1$.
\end{proof}

\section{Application: an electrostatic model}\label{Section-Electrostatic}

The theory presented in the previous sections leads to interesting electrostatic models.
For background on electrostatics of OP we refer the reader to the books by Szeg\"{o} \cite[pp.\;140--142]{Szego}
and Ismail \cite[Chapter\;3]{Ismail}, and the articles by Ismail \cite{Ismail-e1,Ismail-e2}
and Marcell\'an et. al. \cite{Paco-e1}.
Fix an integer number $k$, with $k\geq3$, and let $n$ be a multiple of $k$,
so there exists $\ell\in\mathbb{N}$ such that
$$
n=k\ell\;.
$$
Suppose that $n$ unit charges at points $x_1<x_2<\ldots<x_n$ are distributed on the set
$(-1,1)\setminus Z_{U_{k-1}}$, where
$$Z_{U_{k-1}}:=\Big\{\cos\frac{j\pi}{k}\,|\,j=1,2,\ldots,k-1\Big\}$$
is the set of zeros of the Chebyshev polynomial of the second kind of degree $k-1$,
in such a way that each one of the $k$ open intervals intervals
$\big]-1,\cos\frac{(k-1)\pi}{k}\big[$, $\big]\cos\frac{(k-1)\pi}{k},\cos\frac{(k-2)\pi}{k}\big[$, ...\,,
$\big]\cos\frac{\pi}{k},1\big[$
contains precisely $\ell$ points, i.e.,
\begin{equation}\label{k-intervals}
\cos\frac{(k-j)\pi}{k}<x_{j\ell+1}<x_{j\ell+2}<\cdots<x_{(j+1)\ell}<\cos\frac{(k-j-1)\pi}{k}
\end{equation}
for each $j=0,1,\ldots,k-1$.
In addition, assume that both $-1$ and $+1$ have the same charge $q\geq\frac14$,
as well as there are equal charges at each point of $Z_{U_{k-1}}$,
being $\widetilde{q}:=2q-\frac12$ the common charge at each of these points.
Figure \ref{fig1} illustrates the situation.
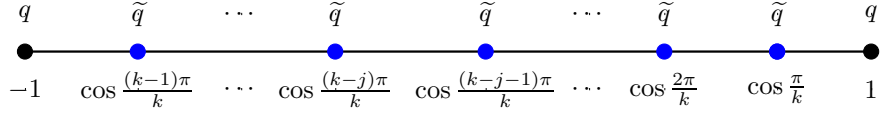
\begin{figure}\begin{center}
\begin{tikzpicture}[xscale=0.5,yscale=0.5] 
\draw [line width=0.3mm, black ] (0,-3) -- (22.5,-3) ;;
\draw[black,fill] (0,-3) circle (0.2cm);;
\draw [line width=0.3mm] (0,-4) -- (0,-4) node [] {$-1$};;
\draw [line width=0.3mm] (0,-2) -- (0,-2) node [] {$q$};;
\draw[blue,fill] (3,-3) circle (0.2cm);;
\draw [line width=0.3mm] (3,-4) -- (3,-4) node [] {$\cos\frac{(k-1)\pi}{k}$};;
\draw [line width=0.3mm] (3,-2) -- (3,-2) node [] {$\widetilde{q}$};;
\draw [line width=0.3mm] (5.75,-4) -- (5.75,-4) node [] {$\cdots$};;
\draw [line width=0.3mm] (5.75,-2) -- (5.75,-2) node [] {$\cdots$};;
\draw[blue,fill] (8.25,-3) circle (0.2cm);;
\draw [line width=0.3mm] (8.25,-4) -- (8.25,-4) node [] {$\cos\frac{(k-j)\pi}{k}$};;
\draw [line width=0.3mm] (8.25,-2) -- (8.25,-2) node [] {$\widetilde{q}$};;
\draw[blue,fill] (12.25,-3) circle (0.2cm);;
\draw [line width=0.3mm] (12.25,-4) -- (12.25,-4) node [] {$\cos\frac{(k-j-1)\pi}{k}$};;
\draw [line width=0.3mm] (12.25,-2) -- (12.25,-2) node [] {$\widetilde{q}$};;
\draw [line width=0.3mm] (15,-4) -- (15,-4) node [] {$\cdots$};;
\draw [line width=0.3mm] (15,-2) -- (15,-2) node [] {$\cdots$};;
\draw[blue,fill] (17,-3) circle (0.2cm);;
\draw [line width=0.3mm] (17,-4) -- (17,-4) node [] {$\cos\frac{2\pi}{k}$};;
\draw [line width=0.3mm] (17,-2) -- (17,-2) node [] {$\widetilde{q}$};;
\draw[blue,fill] (20,-3) circle (0.2cm);;
\draw [line width=0.3mm] (20,-4) -- (20,-4) node [] {$\cos\frac{\pi}{k}$};;
\draw [line width=0.3mm] (20,-2) -- (20,-2) node [] {$\widetilde{q}$};;
\draw[blue,fill] (20,-3) circle (0.2cm);;
\draw [line width=0.3mm] (22.5,-4) -- (22.5,-4) node [] {$1$};;
\draw [line width=0.3mm] (22.5,-2) -- (22.5,-2) node [] {$q$};;
\draw[black,fill] (22.5,-3) circle (0.2cm);
\end{tikzpicture}
\caption{\small The charges in the fixed positions in the electrostatic model}
\label{fig1}
\end{center}
\end{figure}
All these charges interact and repel each other according to the law of logarithm potential.
The energy of these electrostatic charges is therefore represented by
\begin{equation}\label{Energy}
E(x_1,\ldots,x_n):=-\sum_{i=1}^n\sum_{\substack{j=1\\j\neq i}}^n\ln|x_i-x_j|
-2q\sum_{i=1}^n\ln(1-x_i^2)-2\widetilde{q}\sum_{i=1}^{n}\ln\big|\widehat{U}_{k-1}(x_i)\big|
\end{equation}
We regard $E$ as a function defined on the $n-$dimensional cube $[-1,1]^n$, and so
$$
E(x_1,\cdots,x_n)=+\infty\quad\mbox{\rm if}\quad (x_1,\cdots,x_n)\in\Xi\;,
$$
where $\Xi:=\Lambda\cup\big\{(x_1,\cdots,x_n)\in[-1,1]^n\,|\,\mbox{$x_i\in\{-1,1\}\cup Z_{U_{k-1}}$ for some $i$}\big\}$,
and
$$
\Lambda:=\big\{(x_1,\cdots,x_n)\in\mathbb{R}^n\,|\,\mbox{\rm $x_i=x_j$ for some pair $(i,j)$, with $i\neq j$}\,\big\}\;.
$$
The local minima of $E(x_1,\cdots,x_n)$ correspond to the electrostatic equilibrium.
These minima cannot be attained at points of the set $\Xi$
(since $E=+\infty$ on $\Xi$). Therefore for finding the points where $E$ attains minima,
we may regard $E$ as a function defined on the open set
$\Omega:=\Sigma\setminus\Lambda$, where $\Sigma$ is the $n-$dimensional open rectangle
$$
\Sigma:=\prod_{j=1}^{k}\Big]\cos\frac{(k-j+1)\pi}{k},\cos\frac{(k-j)\pi}{k}\Big[^{\,\ell}\;.
$$
In order to find the minimum of $E$ we need to solve the system of equations
\begin{equation}\label{partialL}
\frac{\partial E}{\partial x_\nu}=0\;,\quad \nu=1,2,\ldots,n\;.
\end{equation}
Using the relation $\;\widehat{U}_{k-1}'(x)/\widehat{U}_{k-1}(x)
=\sum_{j=1}^{k-1}1/\big(x-\cos\frac{j\pi}{k}\big)\,$,
we compute
\begin{equation}\label{partialLxnu}
\frac{\partial E}{\partial x_\nu}=
-2\sum_{\substack{i=1\\i\neq \nu}}^n\frac{1}{x_\nu-x_i}
-4q\frac{x_\nu}{x_\nu^2-1}-2\widetilde{q}\sum_{j=1}^{k-1}\frac{1}{x_\nu-\cos\frac{j\pi}{k}}
\end{equation}
for each $\nu=1,2,\ldots,n$. Therefore, setting
$\;p_n(x):=(x-x_1)(x-x_2)\cdots(x-x_n)\,$,
we see that (\ref{partialL}) can be rewritten as
\begin{equation}\label{pxDeriv}
\frac{p_n^{''}(x_\nu)}{p_n^{'}(x_\nu)}=
-\frac{2q}{x_\nu-1}-\frac{2q}{x_\nu+1}
-\sum_{j=1}^{k-1}\frac{2\widetilde{q}}{x_\nu-\cos\frac{j\pi}{k}}\;,\quad 1\leq\nu\leq n\;.
\end{equation}
In the next theorem we show that the electrostatic equilibrium can
be described in terms of the zeros of the sieved ultraspherical polynomial
of the first kind $c_n^{\lambda}(\cdot;k)$, for an appropriate choice of $\lambda$.

\begin{theorem}\label{Thm-electrost-sieved}
Let $k,\ell\in\mathbb{N}$, being $k\geq3$, and let $n:=k\ell$. 
Then the energy $(\ref{Energy})$ of the system with $n$ unit charges at
$x_1<x_2<\cdots< x_n$ on $[-1,1]$ subject to condition $(\ref{k-intervals})$, 
with charges $q\geq\frac14$ at the points $\pm1$ and charges
$\widetilde{q}:=2q-\frac12$ at the points $\cos\frac{j\pi}{k}$, $1\leq j\leq k-1$,
is minimal when $x_1,\ldots,x_n$ are the zeros of the sieved ultraspherical polynomial
of the first kind $c_n^{\lambda}(x;k)\equiv \frac{\ell !}{(2\lambda)_\ell}\,C_\ell^{\lambda}\big(T_k(x)\big)$,
where $\lambda:=2q-\frac12$. Moreover, the equilibrium position is unique.
\end{theorem}

\begin{proof}
Let $\{x_{n,\nu}^\lambda\}_{\nu=1}^n$ be the set of zeros of $c_n^{\lambda}(x;k)$.
According to Theorem \ref{cn-JnKnLn}, this polynomial fulfills
the second order linear ODE (\ref{ED2O}), hence
evaluating at each zero $x_{n,\nu}^\lambda$, we obtain
$$
J_{n}\big(x_{n,\nu}^\lambda\big)\big\{c_{n}^\lambda\big\}^{''}\big(x_{n,\nu}^\lambda;k\big)
+K_{n}\big(x_{n,\nu}^\lambda\big)\big\{c_{n}^\lambda\big\}^{'}\big(x_{n,\nu}^\lambda;k\big)=0\;,
\quad 1\leq\nu\leq n\;.
$$
Therefore, taking into account (\ref{JKL1}), we deduce
\begin{equation}\label{Bnk2deriv}
\frac{\big\{c_{n}^\lambda\big\}^{''}\big(x_{n,\nu}^\lambda;k\big)}
{\big\{c_{n}^\lambda\big\}^{'}\big(x_{n,\nu}^\lambda;k\big)}=
\frac{M_{n}^{'}\big(x_{n,\nu}^\lambda\big)}{M_{n}\big(x_{n,\nu}^\lambda\big)}
-\frac{\Psi\big(x_{n,\nu}^\lambda\big)}
{\Phi\big(x_{n,\nu}^\lambda\big)}\;,\;\;1\leq\nu\leq n\;,
\end{equation}
where $\Phi$ and $\Psi$ are given by (\ref{PhiPsi-cn}),
and $M_{n}\equiv M_{k\ell}$ and $N_{n}\equiv N_{k\ell}$ are given by (\ref{MNk-1c}).
Next we will show that
\begin{equation}\label{Psi-div-Phi}
\frac{\Psi(x)}{\Phi(x)}=
\frac{2\lambda+1}{2}\,\left(\frac{1}{x-1}+\frac{1}{x+1}\right)
+(2\lambda+1)\sum_{j=1}^{k-1}\frac{1}{x-\cos\frac{j\pi}{k}}\;.
\end{equation}
Indeed, by (\ref{PhiPsi-cn}) and taking into account the last relation in (\ref{CH12345}),
we deduce
$$
\frac{\Psi(x)}{\Phi(x)}=-k(2\lambda+1)\frac{x}{1-x^2}
+\frac{k(2\lambda+1)}{2}\frac{1}{1-x^2}\frac{\widehat{U}_{k-2}(x)}{\widehat{U}_{k-1}(x)}\;.
$$
Thus (\ref{Psi-div-Phi}) follows by straightforward computations using the relations
$$
\frac{x}{1-x^2}=\frac12\frac{1}{1-x}-\frac12\frac{1}{1+x}\;,\quad
\frac{\widehat{U}_{k-2}(x)}{\widehat{U}_{k-1}(x)}
=\frac{2}{k}\sum_{j=1}^{k-1}\frac{\sin^2\frac{j\pi}{k}}{x-\cos\frac{j\pi}{k}}\;,
$$
$$
\frac{\sin^2\frac{j\pi}{k}}{(1-x^2)\big(x-\cos\frac{j\pi}{k}\big)}=
\frac{\cos^2\frac{j\pi}{2k}}{1-x}-
\frac{\sin^2\frac{j\pi}{2k}}{1+x}+
\frac{1}{x-\cos\frac{j\pi}{k}}\;,
$$
$$
\sum_{j=1}^{k-1}\cos^2\frac{j\pi}{2k}=\sum_{j=1}^{k-1}\sin^2\frac{j\pi}{2k}=\frac{k-1}{2}\;.
$$
On another hand, using (\ref{MNk-1c}) we have
\begin{equation}\label{MnM1n}
\frac{M_n'(x)}{M_n(x)}=\frac{\widehat{U}_{k-1}'(x)}{\widehat{U}_{k-1}(x)}
=\sum_{j=1}^{k-1}\frac{1}{x-\cos\frac{j\pi}{k}}\;.
\end{equation}
Combining (\ref{Psi-div-Phi}) and (\ref{MnM1n}) we obtain
\begin{equation}\label{Psi-div-Phi-Mn}
\frac{M_n'(x)}{M_n(x)}-\frac{\Psi(x)}{\Phi(x)}=
-\frac{2\lambda+1}{2}\left(\frac{1}{x-1}+\frac{1}{x+1}\right)
-2\lambda\sum_{j=1}^{k-1}\frac{1}{x-\cos\frac{j\pi}{k}}\;.
\end{equation}
Finally, from (\ref{Bnk2deriv}) and (\ref{Psi-div-Phi-Mn}) we deduce
\begin{equation}\label{Bnk2deriv1}
\frac{\big\{c_{n}^\lambda\big\}^{''}\big(x_{n,\nu}^\lambda;k\big)}
{\big\{c_{n}^\lambda\big\}^{'}\big(x_{n,\nu}^\lambda;k\big)}=
-\frac{2q}{x_{n,\nu}^\lambda-1}-\frac{2q}{x_{n,\nu}^\lambda+1}
-\sum_{j=1}^{k-1}\frac{2\widetilde{q}}{x_{n,\nu}^\lambda-\cos\frac{j\pi}{k}}\;,\quad 1\leq\nu\leq n\;.
\end{equation}
Therefore, the zeros $x_\nu\equiv x_{n,\nu}^\lambda$ of $c_n^{\lambda}(x;k)$ solve the system of equations (\ref{partialL}),
i.e., $\mbox{\rm x}^*:=(x_{n,1}^\lambda,\ldots,x_{n,n}^\lambda)$ is a critical point of $E$.
Notice that $\mbox{\rm x}^*\in\Omega$, i.e., $\mbox{\rm x}^*$ fulfills (\ref{k-intervals}),
since each zero $x_{n,\nu}^\lambda$ of $c_n^{\lambda}(x;k)$ satisfies $C_\ell^{\lambda}(T_k(x_{n,\nu}^\lambda)\big)=0$
($1\leq\nu\leq n$) and it is well known (and easy to check) that the ultraspherical polynomial $C_\ell^{\lambda}$
has $\ell$ distinct zeros in $]-1,1[$
and the Chebyshev polynomial $T_k$ has its critical points at the zeros of $U_{k-1}$, being the absolute value of $T_k$ at each critical point equal to $1$.
Next we show that $\mbox{\rm x}^*$ is indeed a (local) minimum of $E$, and, moreover,
it is the unique (global) minimum of $E$.
We will argue as in the proof of \cite[Theorem\;2.1]{Ismail-e1}.
Indeed, consider the hessian matrix $H(\mbox{\rm x})=[h_{i,j}(\mbox{\rm x})]_{i,j=1}^n$,
$h_{i,j}(\mbox{\rm x}):=\partial E(\mbox{\rm x})/\partial x_i\partial x_j$, $\mbox{\rm x}\equiv(x_1,\ldots,x_n)$.
Taking into account (\ref{partialLxnu}), we compute
$$
h_{i,j}(\mbox{\rm x})=\left\{
\begin{array}{lcl}
\displaystyle \frac{-2}{(x_i-x_j)^2} & \mbox{\rm if} & i\neq j \,, \\ [1em]
\displaystyle \sum_{\substack{\nu=1\\ \nu\neq i}}^n\frac{2}{(x_i-x_\nu)^2}+4q\,\frac{x_i^2+1}{(x_i^2-1)^2}+ \sum_{\substack{\nu=1}}^{k-1}\frac{2\widetilde{q}}{\big(x_i-\cos\frac{\nu\pi}{k}\big)^2}& \mbox{\rm if} & i= j \;.
\end{array}
\right.
$$
Therefore, for each $\mbox{\rm x}\in\Omega$, $H(\mbox{\rm x})$ is a real symmetric matrix with positive diagonal elements
and strictly diagonally dominant
(i.e., $|h_{i,i}(\mbox{\rm x})|>\sum_{\nu=1,\nu\neq i}^n|h_{i,\nu}(\mbox{\rm x})|$ for each $i=1,2,\ldots,n$).
It follows from \cite[Theorem\;6.1.10-(c)]{HJ1992}  
that $H(\mbox{\rm x})$ is a positive definite matrix for each $\mbox{\rm x}\in\Omega$, hence  $\mbox{\rm x}^*$
is indeed a local minimum of $E$ in $\Omega$.
To see that this minimum is unique (and so it is a global minimum), we may argue as in \cite[p.\;140]{Szego},
using the arithmetic-geometric mean inequality.
Indeed, notice first that a point in $\Omega$ is a minimum of $E$ if and only if it is
a maximum of $T:[-1,1]^n\to\mathbb{R}$ defined by
$$
T(x_1,\ldots,x_n):=\exp\big(-E(x_1,\ldots,x_n)\big)\;,
$$
or, explicitly,
$$
T(x_1,\ldots,x_n):=
\prod_{r=1}^n(1-x_r^2)^{2q}\cdot
\prod_{s=1}^n\prod_{t=1}^{k-1}\big|x_s-\mbox{$\cos\frac{t\pi}{k}$}\big|^{2\widetilde{q}}\cdot
\prod_{\substack{\nu,\mu=1\\ \nu<\mu}}^n\big|x_\nu-x_\mu\big|^2\;.
$$
Suppose that $T$ attains relative maxima at two different critical points
$\mbox{\rm x}=(x_1,\ldots,x_n)\in\Omega$ and $\mbox{\rm x}'=(x_1',\ldots,x_n')\in\Omega$.
These points fulfill (\ref{k-intervals}), i.e.,
$$
\cos\frac{(k-j)\pi}{k}<x_{j\ell+1}<x_{j\ell+2}<\cdots<x_{(j+1)\ell}<\cos\frac{(k-j-1)\pi}{k}\;,
$$
$$
\cos\frac{(k-j)\pi}{k}<x_{j\ell+1}'<x_{j\ell+2}'<\cdots<x_{(j+1)\ell}'<\cos\frac{(k-j-1)\pi}{k}
$$
for each $j=0,1,\ldots,k-1$. Then, considering
$\mbox{\rm x}'':=(\mbox{\rm x}+\mbox{\rm x}')/2\equiv(x_1'',\ldots,x_n'')$,
we deduce, for each $\nu,\mu=1,2,\ldots,n$ and $t=0,1,\cdots,k$,
$$
\begin{array}{c}\displaystyle
|x_\nu''-x_\mu''|=\frac{|x_\nu-x_\mu|+|x_\nu'-x_\mu'|}{2}\geq\big|x_\nu-x_\mu\big|^{1/2}\big|x_\nu'-x_\mu'\big|^{1/2}\;, \\ [1em]
\big|x_\nu''-\mbox{$\cos\frac{t\pi}{k}$}\big|\geq
\big|x_\nu-\mbox{$\cos\frac{t\pi}{k}$}\big|^{1/2}\big|x_\nu'-\mbox{$\cos\frac{t\pi}{k}$}\big|^{1/2}\;,
\end{array}
$$
and so $T(\mbox{\rm x}'')\geq T(\mbox{\rm x})^{1/2}\,T(\mbox{\rm x}')^{1/2}$.
Therefore, assuming without loss of generality that
$\min\{T(\mbox{\rm x}),T(\mbox{\rm x}')\}=T(\mbox{\rm x})$, we obtain $T(\mbox{\rm x}'')\geq T(\mbox{\rm x})$.
Proceeding in the same way, taking $\mbox{\rm x}''':=(\mbox{\rm x}+\mbox{\rm x}'')/2$,
we see that $\mbox{\rm x}'''\in\Omega$ and $T(\mbox{\rm x}''')\geq T(\mbox{\rm x})$.
Continuing the process we obtain a sequence of different points $\mbox{\rm x}^{(m)}\in\Omega$
such that $\mbox{\rm x}^{(m)}\to\mbox{\rm x}$ (as $m\to\infty$)
and $T(\mbox{\rm x}^{(m)})\geq T(\mbox{\rm x})$ for each $m\in\mathbb{N}$.
Therefore, since $T$ attains a relative maximum at $\mbox{\rm x}$ then there exists an order $m_0$ such that
$T(\mbox{\rm x}^{(m)})=T(\mbox{\rm x})$ for each $m\geq m_0$, and so $T$ attains also a relative maximum
at each point $\mbox{\rm x}^{(m)}$ with $m\geq m_0$.
As a consequence, every neighborhood of $\mbox{\rm x}$
contains critical points of $T$ (in $\Omega$) different from $\mbox{\rm x}$.
However, this is impossible, since the hessian matrix $H(\mbox{\rm x})$ is invertible at each critical point
(since it is strictly diagonally dominant; see \cite[Theorem\;6.1.10-(a)]{HJ1992})
hence each critical point is nondegenerate, and so an isolated critical point
(see e.g. \cite[Section 16.5, Problem 4-(b)]{JDieudonne}, or \cite[Chapter 3, Section 8, Theorems 4 and 5]{ELLima}).
\end{proof}

\begin{remark}
Since $n=k\ell$, then 
$B_{n+k-1}^{\lambda-1}(x;k)/U_{k-1}(x)=\frac{(2\lambda)_\ell}{\ell!}\,c_n^{\lambda}(x;k)$,
hence the electrostatic problem under consideration
in Theorem \ref{Thm-electrost-sieved} is equally solved by the zeros of
the sieved ultraspherical polynomial of the second kind $B_{n+k-1}^{\lambda-1}(x;k)$ which are different
from the zeros of $U_{k-1}(x)$. 
All these polynomials are plotted in Figure \ref{fig2} for $\lambda=\frac32$,
$k=5$ and $n=10$ (or $\ell=2$).
\end{remark}

\begin{remark}
The analysis of the electrostatic problem whenever
the residue modulo $k$ of $n$ is different from $0$ remains an open problem.
\end{remark}

\begin{center}
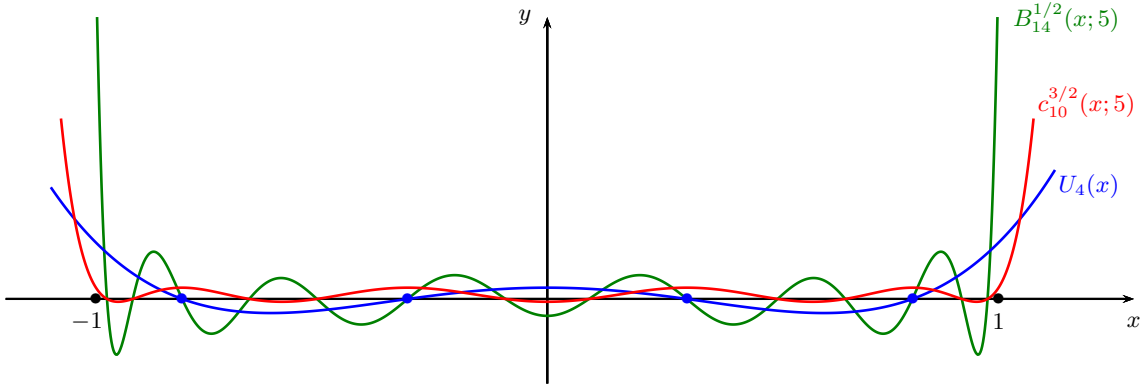
\begin{figure}[!h]
\begin{pspicture}(0,0)(-6,4)
\psset{xunit=6cm,yunit=0.15cm} 
 \psset{plotpoints=1000}
\psaxes[Dx=2,Dy=100]{->}(0,0)(-1.2,-7.5)(1.3,25) 
\infixtoRPN{30720*x^14-99840*x^12+126720*x^10-79200*x^8+25200*x^6-3774*x^4+411/2*x^2-3/2} \psplot[linecolor=verde,linewidth=1pt]{-0.9975}{0.9975}{\RPN}
\infixtoRPN{16*x^4-12*x^2+1} \psplot[linecolor=blue,linewidth=1pt]{-1.1}{1.125}{\RPN}
\infixtoRPN{320*x^10-800*x^8+700*x^6-250*x^4+125/4*x^2-1/4} \psplot[linecolor=red,linewidth=1pt]{-1.0775}{1.0775}{\RPN}
\rput(-1,0){{\black$\bullet$}}
\rput(1,0){{\black$\bullet$}}
\rput(-0.31,0){{\blue$\bullet$}}
\rput(0.31,0){{\blue$\bullet$}}
\rput(-0.81,0){{\blue$\bullet$}}
\rput(0.81,0){{\blue$\bullet$}}
\rput(1.2,10){\blue\begin{small}$U_4(x)$\end{small}}
\rput(1.15,25){\verde\begin{small}$B_{14}^{1/2}(x;5)$\end{small}}
\rput(1.195,17.5){\red\begin{small}$c_{10}^{3/2}(x;5)$\end{small}}
\rput(-1.02,-2){\begin{small}$-1$\end{small}}
\rput(1,-2){\begin{small}$1$\end{small}}
\rput(1.3,-2){\begin{small}$x$\end{small}}
\rput(-0.05,25){\begin{small}$y$\end{small}}
\end{pspicture}
\vspace{2.5em}
\caption{\small Plots of the polynomials involved in the electrostatic model
for the choices $\lambda=3/2$, $k=5$, and $n=10$}
\label{fig2}
\end{figure}
\end{center}

\section*{Acknowledgements}

KC and JP are supported by the Centre for Mathematics of the University of Coimbra--UID/MAT/00324/2019, funded by the Portuguese Government through FCT/MEC and co-funded by the European Regional Development Fund through the Partnership Agreement PT2020.
MNJ supported by UID/Multi/04016/2019, funded by FCT.
MNJ also thanks the Instituto Polit\'e\-cnico de Viseu and CI\&DETS for their support.

\end{document}